\documentclass{article}

\title{The orbital counting problem for hyperconvex representations}
\author{A. Sambarino}
\date{}

\usepackage[titletoc,toc,title]{appendix}
\usepackage{marvosym}
\usepackage{psfrag}
\usepackage[english]{babel}
\usepackage{mathrsfs}
\usepackage{verbatim}
\usepackage[ansinew]{inputenc}
\usepackage{amsmath,amsthm,amssymb,amscd}
\usepackage[all,cmtip]{xy}
\usepackage{xspace}
\usepackage[dvips]{color}
\usepackage{epsfig}
\usepackage{pb-diagram}
\usepackage{amsfonts}
\usepackage{graphicx}

\newcommand{\R}{\mathbb{R}}

\newcommand{\N}{\mathbb{N}}

\renewcommand{\P}{\mathbb{P}}

\renewcommand{\/}{\backslash}
\renewcommand{\k}{\kappa}

\newcommand{\eps}{\varepsilon}

\newcommand{\vacio}{\emptyset}

\newcommand{\G}{\Gamma}
\renewcommand{\l}{\ell}
\newcommand{\<}{\left<}
\renewcommand{\>}{\right>}
\newcommand{\E}{\Sigma}
\newcommand{\scr}{\mathscr}
\newcommand{\g}{\gamma}
\renewcommand{\a}{\alpha}

\newcommand{\z}{\zeta}
\newcommand{\w}{\widetilde}

\newcommand{\bord}{\partial_\infty}
\newcommand{\vo}[1]{\overline{#1}}
\newcommand{\sub}[1]{\underline{#1}}

\newcommand{\ta}{\Theta}
\newcommand{\vect}{\beta}
\newcommand{\cone}{\scr L}
\newcommand{\Pat}{\chi}
\newcommand{\Pats}{\Omega}
\newcommand{\bus}{\sigma}
\newcommand{\buss}{\sigma}

\newcommand{\grupo}{\Delta}
\newcommand{\om}{\omega}

\newcommand{\LL}{\lambda}
\newcommand{\posgen}{\scr F^{(2)}}
\newcommand{\vov}{}%[1]{\overline{#1}}

\newcommand{\p}{{\mathsf{f}}}
\newcommand{\zz}{{\mathsf{Z}}}
\newcommand{\wk}{\check}
\newcommand{\BM}{\w{\Pat}_{\ta}}

\newcommand{\Gromov}[2]{[#1,#2]}
\newcommand{\dens}{\mu}
\newcommand{\simbolo}[1]{\stackrel{#1}{\sim}}
\newcommand{\mm}[1]{[#1]}
\newcommand{\F}{\overline{X}_F}

\newcommand{\uu}[1]{{\mathsf{u}}_{#1}}
\newcommand{\II}{I}
\newcommand{\ca}{c}

\renewcommand{\t}{\theta}
\renewcommand{\aa}{\underline{a}}

\newcommand{\UU}{{u_{\rho(\G)}}}
\newcommand{\TT}{{\ta_{\rho}}}

\newcommand{\Gr}{\mathscr G}

\renewcommand{\1}{\mathbf 1}
\renewcommand{\L}{\Lambda}

\renewcommand{\cal}{\mathcal}

\DeclareMathOperator{\Li}{L}
\DeclareMathOperator{\rank}{rank}
\DeclareMathOperator{\ii}{i}

\DeclareMathOperator{\tran}{tr}

\DeclareMathOperator{\inte}{int}
\DeclareMathOperator{\vol}{vol}

\DeclareMathOperator{\SL}{SL}
\DeclareMathOperator{\PSL}{PSL}

\DeclareMathOperator{\Leb}{Leb}

\DeclareMathOperator{\clase}{C}

\DeclareMathOperator{\id}{id}

\DeclareMathOperator{\PGL}{PGL}

\DeclareMathOperator{\lie}{Lie}

\newtheorem*{teos}{Theorem}
\newtheorem*{teo1}{Theorem A}
\newtheorem*{teoB}{Theorem B}

\newtheorem{teo}{Theorem}[section]
\newtheorem{cor}[teo]{Corollary}
\newtheorem*{cors}{Corollary}

\newtheorem{lema}[teo]{Lemma}

\newtheorem{prop}[teo]{Proposition}

\theoremstyle{definition}

\newtheorem{afi}[teo]{Claim}

\newtheorem{defi}[teo]{Definition}

\newtheorem{obs}[teo]{Remark}

\theoremstyle{remark}

\let\oldsqrt\sqrt
\def\sqrt{\mathpalette\DHLhksqrt}
\def\DHLhksqrt#1#2{%
\setbox0=\hbox{$#1\oldsqrt{#2\,}$}\dimen0=\ht0
\advance\dimen0-0.2\ht0
\setbox2=\hbox{\vrule height\ht0 depth -\dimen0}%
{\box0\lower0.4pt\box2}}

\begin{document}
\maketitle

\begin{abstract}We give a precise counting result on the symmetric space of a noncompact real algebraic semisimple group $G,$ for a class of discrete subgroups of $G$ that contains, for example, representations of a surface group on $\PSL(2,\R)\times\PSL(2,\R),$ induced by choosing two points on the Teichm\"uller space of the surface; and representations on the Hitchin component of $\PSL(d,\R).$ We also prove a mixing property for the Weyl chamber flow in this setting.
\end{abstract}

%\tableofcontents

\section{Introduction}

The \emph{Orbital Counting Problem} is: given a discrete subgroup $\grupo$ of a connected noncompact real algebraic semisimple Lie group $G,$ find an asymptotic for the growth of $$\#\{g\in\grupo:d_X(o,g\cdot o)\leq t\}$$ as $t\to \infty,$ where $o=[K]$ is a basepoint on $X=G/K,$ the symmetric space of $G,$ endowed with a $G$-invariant Riemannian metric.

When the group $\grupo$ is a lattice, this problem has been studied by Eskin-McMullen \cite{esk}. They prove that the number of points in $\grupo\cdot o\cap B(o,t),$ is equivalent (modulo a constant) to the volume $\vol(B(o,t))$ of the ball of radius $t.$ Hence, the asymptotic has a polynomial term together with an exponential term. Similar results have been obtained by Duke-Rudnick-Sarnak \cite{drs}.

We will hence focus on subgroups of infinite covolume. An important tool for such groups, in negative curvature, is the limit set of the group on the visual boundary of the space in consideration. On higher rank, it turns out to be more useful to consider the Furstenberg boundary.

Let $P$ be a minimal parabolic subgroup of $G,$ and denote by $\scr F_G=\scr F=G/P$ the \emph{Furstenberg boundary} of $X.$ Benoist \cite{limite} has shown that the action of $\grupo$ on $\scr F$ has a smallest closed invariant set, called the \emph{limit set of $\grupo$ on $\scr F$}, and denoted by $\Li_\grupo.$ 

The limit set is well understood for Schottky groups. These are finitely generated free subgroups of $G,$ for which one has a good control on the relative position of the fixed points on $\scr F$ of the free generators, together with nice contraction properties.

This precise information allows Quint \cite{quint4} to build an equivariant continuous map, from the boundary at infinity of the group into $\scr F.$ The limit set is hence identified with a subshift of finite type. Quint \cite{quint4} uses the Thermodynamic Formalism on this subshift, to obtain an exponential equivalence for the orbital counting problem.

This work consists in studying the orbital counting problem, for a class of subgroups called hyperconvex representations, which we will now define.

The product $\scr F\times\scr F$ has a unique open $G$-orbit, denoted by $\posgen.$ For example, when $G=\PGL(d,\R),$ the space $\scr F$ is the space of complete flags of $\R^d,$ i.e. families of subspaces $\{V_i\}_{i=0}^d$ such that $V_i\subset V_{i+1}$ and $\dim V_i=i;$ and the set $\posgen$ is the set of flags in general position, i.e. pairs $(\{V_i\},\{W_i\})$ such that, for every $i,$ one has $$V_i\oplus W_{d-i}=\R^d.$$

Let $\G$ be the fundamental group of a closed connected negatively curved Riemannian manifold (for any basepoint).

\begin{defi}  We say that a representation $\rho:\G\to G$ is \emph{hyperconvex}, if there exists a H\"older-continuous $\rho$-equivariant map $\z:\bord\G\to\scr F,$ such that the pair $(\z(x),\z(y))$ belongs to $\posgen$ whenever $x,y\in\bord\G$ are distinct.
\end{defi}

If $G$ is a rank 1 simple group, then its Furstenberg boundary is the visual boundary of the symmetric space, and the open orbit $\posgen$ is $$\{(x,y)\in\scr F\times\scr F: x\neq y\}.$$ The classical Morse's Lemma implies thus, that a quasi-isometric embeding $\G\to G,$ is a hyperconvex representation.

Hyperconvex representations where introduced by Labourie \cite{labourie}, in his study of the Hitchin component. Consider a closed connected oriented surface $\E,$ of genus $g\geq2,$ and say that a representation $\pi_1(\E)\to\PSL(d,\R)$ is \emph{Fuchsian}, if it factors as $$\pi_1(\E)\to\PSL(2,\R)\to\PSL(d,\R),$$ where $\PSL(2,\R)\to\PSL(d,\R)$ is induced by the irreducible linear action of $\SL(2,\R)$ on $\R^d$ (unique modulo conjugation by $\SL(d,\R)$), and $\pi_1(\E)\to\PSL(2,\R)$ is discrete and faithful. A \emph{Hitchin component} of $\PSL(d,\R),$ is a connected component of the space $\hom(\pi_1(\E),\PSL(d,\R)),$ containing a Fuchsian representation.

\begin{teos}[Labourie \cite{labourie}]A representation in a Hitchin component of $\PSL(d,\R)$ is hyperconvex.
\end{teos}

Finally, recall that if $G$ and $H$ are noncompact real algebraic semisimple groups, then the Furstenberg boundary of $G\times H,$ is $\scr F_G\times\scr F_H.$ Hence, if $\rho:\G\to G$ and $\eta:\G\to H$ are hyperconvex representations, so is the product $\rho\times \eta:\G\to G\times H.$

Denote by $C(Z),$ the Banach space of real continuous functions on a compact space $Z$ (with the uniform topology), and by $C^*(Z)$ its topological dual. Denote by $\F$ the Furstenberg compactification of $X$ (see Section \ref{section:bowenmargulis}). %Recall that every $\g\in\G-\{e\}$ has exactly two fixed points on $\bord\G,$ denoted by $\g_+$ and $\g_-.$ The main result of this work is the following distribution theorem, inspired in Roblin \cite{roblin}.

\begin{teo1}[See Section \ref{section:bowenmargulis}] Let $\rho:\G\to G$ be a Zariski-dense hyperconvex representation. Then there exist $h,c>0,$ and a probability measure $\mu$ on $\F,$ such that $$ce^{-ht}\sum_{\g\in\G-\{e\}:d_X(o,\rho(\g)\cdot o)\leq t}\delta_{\rho(\g)\cdot o} \otimes \delta_{\rho(\g^{-1})\cdot o} \to\mu\otimes\mu,$$ for the weak-star convergence on $C^*(\F^2),$ as $t\to\infty.$
\end{teo1}

Considering the constant function equal to 1, one obtains the following corollary.

\begin{cors}Let $\rho:\G\to G$ be a Zariski-dense hyperconvex representation. Then there exist $h,c>0,$ such that $$ce^{-ht}\#\{\g\in\G:d_X(o,\rho(\g)\cdot o)\leq t\}\to 1,$$ as $t\to\infty.$
\end{cors}

The exponential growth rate $h$ in Theorem A is explicit: it is the topological entropy of a natural flow we construct, associated to the representation $\rho.$ On the contrary, not much information is known about the constant $c.$

As first shown by Margulis \cite{margulistesis} in negative curvature, in order to obtain a counting theorem one usually proves a mixing property of a well chosen dynamical system. In compact manifolds with negative curvature, the geodesic flow plays this role. In infinite covolume, for example for convex cocompact groups, one should restrict the geodesic flow to its nonwandering set. When $\grupo$ is a lattice in higher rank, Eskin-McMullen \cite{esk} use the mixing property of the Weyl chamber flow, to prove the counting result previously mentioned. 

Let $\tau$ be the Cartan involution on $\frak g=\lie(G),$ whose fixed point set is the Lie algebra of $K.$ Consider $\frak p=\{v\in\frak g: \tau v=-v\}$ and $\frak a,$ a maximal abelian subspace contained in $\frak p.$ Denote by $\frak a^+$ a closed Weyl chamber, and $M$ the centralizer of $\exp(\frak a)$ on $K.$ The \emph{Weyl chamber flow} is the right action by translations of $\exp(\frak a)$ in $$\grupo\/G/M.$$ When $\grupo$ is a lattice on $G,$ the mixing property of this action is due to Howe-Moore \cite{hw}.

In this article, we prove a mixing property of the Weyl chamber flow for hyperconvex representations. Before stating the result, let us recall the Patterson-Sullivan Theory on higher rank.

Consider a $G$-invariant Riemannian metric in $X,$ and $\|\ \|$ the induced Euclidean norm on $\frak a,$ invariant under the Weyl group. Consider the Cartan decomposition $G=K\exp(\frak a^+)K,$ and $a:G\to\frak a^+$ the Cartan projection, then for every $g\in G,$ one has $\|a(g)\|=d_X([K],g[K]).$ Hence, one is interested in understanding the growth of $$\#\{g\in\grupo:\|a(g)\|\leq t\},$$ as $t\to\infty.$ Given an open cone $\scr C$ in $\frak a^+,$ consider the exponential growth rate $$h_{\scr C}= \limsup_{s\to \infty} \frac{\log \#\{g\in \grupo: a(g)\in \scr C,\;\|a(g)\|\leq s\}}{s}.$$ The \emph{growth indicator} of $\grupo,$ introduced by Quint \cite{quint2}, is the map $\psi_\grupo:\frak a\to\R\cup\{-\infty\},$ defined by $$\psi_\grupo(v)=\|v\|\inf h_{\scr C},$$ where the greatest lower bound is taken over all open cones containing $v.$ Remark that $\psi_\grupo$ is homogeneous.

Benoist \cite{limite} has introduced the \emph{limit cone} $\cone_\grupo$ of $\grupo,$ as the closed cone in $\frak a^+$ generated by $\{\lambda(g):g\in\grupo\},$ where $\lambda:G\to\frak a ^+$ is the Jordan projection. Quint \cite{quint2} proves the following theorem. 

\begin{teo}[Quint \cite{quint2}]\label{teo:teoquint} Let $\grupo$ be a Zariski-dense discrete subgroup of $G.$ Then $\psi_\grupo$ is concave, upper semi-continuous and the space $$\{v\in\frak a:\psi_\grupo(v)>-\infty\},$$ is the limit cone $\cone_\grupo.$ Moreover $\psi_\grupo$ is nonnegative on $\cone_\grupo,$ and positive on its interior.
\end{teo}

The growth indicator plays the role, in higher rank, of the critical exponent in negative curvature. Denote by $P$ the minimal parabolic group of $G,$ associated to the choice of $\frak a^+.$ The set $\scr F=G/P$ is $K$-homogeneous, the group $M$ is the stabilizer in $K$ of $[P]\in\scr F.$ The \emph{Busemann cocycle} $\bus:G\times\scr F\to\frak a,$ is defined to verify the equation $$gk=l\exp(\bus(g,kM))n,$$ for every $g\in G$ and $k\in K,$ using Iwasawa's decomposition of $G=K\exp(\frak a) N,$ where $N$ is the unipotent radical of $P.$ 

\begin{teo}[Quint \cite{quint1}]\label{teo:JF-PS} Let $\grupo$ be a Zariski-dense discrete subgroup of $G.$ Then for each linear form $\varphi,$ tangent to $\psi_\grupo$ in a direction in the interior of $\cone_\grupo,$ there exists a probability measure $\nu_\varphi$ on $\scr F,$ supported on $\Li_\grupo,$ such that for every $g\in\grupo$ one has, $$\frac{dg_*\nu_\varphi}{d\nu_\varphi}(x)=e^{-\varphi(\bus(g^{-1},x))}.$$
\end{teo}

The measure $\nu_\varphi$ is called a $\varphi$-Patterson-Sullivan measure of $\grupo.$
Denote by $u_0,$ the unique element of the Weyl group that sends $\frak a^+$ to $-\frak a^+.$ The \emph{opposition involution} $\ii:\frak a\to\frak a$ is defined by $\ii=-u_0.$ One has $\ii( a (g))=a(g^{-1}),$ for every $g\in G,$ and thus $\psi_\grupo\circ\ii=\psi_\grupo.$ Moreover if $\varphi\in\frak a^*$ is tangent to $\psi_\grupo,$ so is $\varphi\circ\ii.$ Hence, in higher rank, Patterson-Sullivan's measures come in pairs.

As in negative curvature, one can use these measures to construct invariant measures for the Weyl chamber flow. Consider the action of $G$ on $\posgen \times \frak a,$ via Busemann's cocycle, defined by $$g(x,y,v)=(gx,gy,v-\bus(g,y)).$$ Denote by $\vo P$ the opposite parabolic subgroup of $P,$ associated to the choice of $\frak a^+,$ the stabilizer on $G$ of the point $([P],[\vo P],0)\in\posgen\times\frak a$ is isomorphic to $M,$ and we get thus an identification $G/M=\posgen\times\frak a.$ This is called \emph{Hopf's parametrization} of $G.$

Using Tits's \cite{tits} representations of $G,$ one can define a vector valued Gromov product $\Gr_\Pi:\posgen\to \frak a$ (see Section \ref{section:convex}) such that, for every $g\in G$ and $(x,y)\in\posgen,$ $$\Gr_\Pi(gx,gy)-\Gr_\Pi(x,y)=-(\ii\circ\bus(g,x)+\bus(g,y)).$$ For a given $\varphi\in\frak a^*$ tangent to $\psi_\grupo,$ the measure $$e^{-\varphi(\Gr_\Pi(\cdot,\cdot))}\nu_{\varphi\circ\ii}\otimes\nu_{\varphi}\otimes \Leb_{\frak a}$$ in $\posgen\times\frak a,$ is thus $\grupo$-invariant and $\frak a$-invariant. Denote by $\Pat_\varphi$ the measure induced on the quotient $\grupo\/G/M.$ We call this measure \emph{the Bowen-Margulis measure} for $\varphi,$ its support is the set $$\grupo\/(\Li_\grupo^{(2)}\times\frak a),$$ where $\Li_\grupo^{(2)} =(\Li_\grupo)^2 \cap\posgen.$ This set is the analogous, in higher rank, of the nonwandering set of the geodesic flow in negative curvature. An important contrast though, is that when $\grupo$ is not a lattice and $G$ is simple (of higher rank), the measure $\Pat_\varphi$ is expected to have infinite total mass. For example, Quint \cite{quintconvexo} has shown that if $\grupo\/(\Li_\grupo^{(2)}\times\frak a)$ is compact, then $\grupo$ is a cocompact lattice.

We prove the following mixing property, for hyperconvex representations, inspired by the work of Thirion \cite{thirionmixing}. He proves an analogous mixing property for ping-pong groups.

\begin{teoB}[Theorem \ref{teo:mixing}] Let $\rho:\G\to G$ be a Zariski-dense hyperconvex representation, and consider $\varphi\in\frak a^*$ tangent to $\psi_\grupo$ in the direction $u_\varphi.$ Then there exists $\kappa>0$ such that, for any two compactly supported continuous functions $f_0,f_1:\rho(\G)\/G/M\to\R,$ one has $$(2\pi t)^{(\rank(G)-1)/2}\Pat_\varphi(f_0\cdot f_1\circ \exp({tu_\varphi}))\to\kappa \Pat_\varphi(f_0) \Pat_\varphi(f_1),$$ as $t\to\infty.$ 
\end{teoB}

In Section \ref{section:cocyclosreales}, we recall results on H\"older cocycles from \cite{quantitative}, of particular interest is the Reparametrizing Theorem \ref{teo:reparametrisation}. This theorem is crucial in understanding the nature of $$\rho(\G)\/ (\Li_{ \rho (\G )}^{(2)} \times \frak a),$$ when $\rho:\G\to G$ is hyperconvex (Proposition \ref{prop:topo}). In Section \ref{section:mixingweyl}, we prove a general mixing property that will imply Theorem B. This is shown in Section \ref{section:convex}. In the last section, we prove Theorem A by adapting a method of Roblin \cite{roblin} and Thirion \cite{thirion1}.

\subsubsection*{Acknowledgments}

The author is extremely thankful to Jean-Fran\c{c}ois Quint for fruitful discussions and guiding, and for proposing (and insisting) to study the orbital counting problem for Hitchin representations. He would like to thank Fran\c{c}ois Ledrappier for discussions, and the referee for very careful reading and considerably improving the exposition of this work.

\section{H\"older cocycles}\label{section:cocyclosreales}

\subsubsection*{Reparametrizations}

The standard reference for the following is Katok-Hasselblat \cite{katokh}. Let $X$ be a compact metric space, $\phi=(\phi_t)_{t\in\R}$ a continuous flow on $X$ without fixed points (i.e. no point in $X$ verifies $\phi_tx=x$ for every $t\in\R$), and $V$ a finite dimensional real vector space.

\begin{defi}A \emph{translation cocycle} over $\phi$ is a map $\k:X\times \R\to V$ that verifies the following two conditions:
\begin{itemize}
\item[-] For every $x\in X$ and $t,s\in\R,$ one has $$\k(x,t+s)=\k(\phi_sx,t)+\k(x,s).$$ \item[-] For every $t\in\R,$ the map $\k(\cdot,t)$ is H\"older-continuous, with exponent independent of $t,$ and with bounded multiplicative constant when $t$ is bounded. 
\end{itemize}
\end{defi}

Two translation cocycles $\k_1$ and $\k_2$ are \emph{Liv\v sic-cohomologous}, if there exists a continuous map $U:X\to V,$ such that for all $x\in X$ and $t\in\R$ one has \begin{equation}\label{eq:algo}\k_1(x,t)-\k_2(x,t)=U(\phi_tx)-U(x).\end{equation}

Denote by $p(\tau)$ the period of a $\phi$-periodic orbit $\tau.$ If $\k$ is a translation cocycle then the \emph{period} of $\tau$ for $\k,$ is defined by $$L_\k(\tau)=\k(x,p(\tau)),$$ for any $x\in\tau.$ It is clear that $L_\k(\tau)$ does not depend on the chosen point $x\in\tau,$ and that the set of periods is a cohomological invariant of $\k.$

The standard example of a translation cocycle, is obtained by considering a H\"older-continuous map $f:X\to V,$ and defining $\k_f:X\times\R\to V$ by \begin{equation}\label{equation:k} \k_f(x,t)=\int_0^tf(\phi_sx)ds.\end{equation} The period, of a periodic orbit $\tau$ for $f,$ is then $$\int_\tau f=\int_0^{p(\tau)}f(\phi_sx)ds.$$

We say that a map $U:X\to V$ is $\clase^1$ \emph{in the direction of the flow} $\phi,$ if for every $x\in X,$ the map $t\mapsto U(\phi_tx)$ is of class $\clase^1,$ and the map $$x\mapsto \left.\frac{\partial }{\partial t}\right|_{t=0}U(\phi_tx)$$ is continuous. Two H\"older-continuous maps, $f,g:X\to V$ are \emph{Liv\v sic-cohomolo\-gous}, if the translation cocycles $\k_f$ and $\k_g$ are. If this is the case, the map $U$ of equation (\ref{eq:algo}) is $\clase^1$ in the direction of the flow, and for all $x\in X$ one has $$f(x)-g(x)=\left.\frac{\partial}{\partial t}\right|_{t=0} U(\phi_tx).$$

If $f:X\to\R$ is positive, then, since $X$ is compact, $f$ has a positive minimum and for every $x\in X,$ the function $\k_f(x,\cdot)$ is an increasing homeomorphism of $\R.$ We then have a map $\a_f:X\times\R\to\R$ that verifies \begin{equation}\label{equation:inversa} \a_f(x,\k_f(x,t))=\k_f(x,\a_f(x,t))=t,\end{equation} for every $(x,t)\in X\times\R.$

\begin{defi}\label{defi:repa}The \emph{reparametrization} of $\phi$ by $f:X\to\R_+^*,$ is the flow $\psi=\psi^f=(\psi_t)_{t\in\R}$ on $X,$ defined by $\psi_t(x)=\phi_{\a_f(x,t)}(x),$ for all $t\in\R$ and $x\in X.$ If $f$ is H\"older-continuous, we will say that $\psi$ is a H\"older reparametrization of $\phi.$
\end{defi}

\begin{obs} If two continuous functions $f,g:X\to\R_+^*$ are Liv\v sic-cohomolo\-gous, then the flows $\psi^f$ and $\psi^g$ are conjugated i.e. there exists a homeomorphism $h:X\to X$ such that, for all $x\in X$ and $t\in\R,$ one has\footnote{This is standard, see \cite[Remark 2.2.]{exponential} for a detailed proof.} $$h(\psi_t^fx) =\psi_t^g(hx).$$
\end{obs}

Denote by $\cal M^{\phi},$ the set of $\phi$-invariant probability measures on $X.$ The \emph{pressure} of a continuous function $f:X\to\R,$ is defined by $$P(\phi,f)=\sup_{m\in\cal M^{\phi}}h(\phi,m)+\int_X fdm,$$ where $h(\phi, m)$ is the metric entropy of $m$ for $\phi.$ A probability measure $m,$ on which the least upper bound is attained, is called an \emph{equilibrium state} of $f.$ An equilibrium state for $f\equiv0$ is called a \emph{probability measure of maximal entropy}, and its entropy is called the \emph{topological entropy} of $\phi,$ denoted by $h_{\textrm{top}}(\phi).$

If $f$ is positive, and $m$ is a $\phi$-invariant probability measure on $X,$ then the probability measure $m^\#,$ defined by \begin{equation}\label{equation:numeral}\frac{dm^\#}{dm}(\cdot)=\frac{f(\cdot)}{\int fdm},\end{equation} is invariant under $\psi^f.$

\begin{lema}[{\cite[Section 2]{quantitative}}]\label{lema:entropia2} If $h=h_{\textrm{top}}(\psi^f)<\infty,$ then the map $m\mapsto m^\#$ is a bijection between equilibrium states of $-hf,$ and probability measures of maximal entropy of $\psi^f.$
\end{lema}

\subsubsection*{Anosov flows and Markov codings}

Assume from now on that $X$ is a compact manifold, and that the flow $\phi$ is $\clase^1.$ We say that $\phi$ is \emph{Anosov}, if the tangent bundle of $X$ splits as a sum of three bundles $$ TX=E^s\oplus E^0\oplus E^u,$$ that are $d\phi_t$-invariant for every $t\in\R$ and, there exist positive constants $C$ and $c$ such that, $E^0$ is the direction of the flow, and for every $t\geq0$ one has $\|d\phi_tv\|\leq Ce^{-ct}\|v\|$ for every $v\in E^s,$ and $\|d\phi_{-t}v\|\leq Ce^{-ct}\|v\|$ for every $v\in E^u.$ 

We need the following classical result of Liv\v sic \cite{livsic2}:

\begin{teo}[Liv\v sic \cite{livsic2}]\label{teo:livsic3} Let $\phi$ be an Anosov flow on $X$ and $\k:X\times \R\to V$ a translation cocycle. If $L_\k(\tau)=0$ for every periodic orbit $\tau,$ then $\k$ is Liv\v sic-cohomologous to $0.$
\end{teo}

As the next lemma proves, one can always chose a translation cocycle of the form $\k_f,$ in the cohomology class of a given translation cocycle $\k.$

\begin{lema}\label{lema:livsicgral}Let $\phi$ be an Anosov flow on $X,$ and let $\k:X\times\R\to V$ be a translation cocycle, then there exists a H\"older-continuous map $f:X\to V,$ such that the cocycles $\k$ and $\k_f$ are Liv\v sic-cohomologous.
\end{lema}

\begin{proof} Fix $C>0,$ and consider the translation cocycle $\k^C,$ defined by $$\k^C(x,t)=\frac 1C\int_0^C \k(\phi_s(x),t)ds.$$ The translation cocycles $\k^C$ and $\k$ are Liv\v sic-cohomolgous since they have the same periods. One easily checks that $\k^C(\cdot,t)$ is of class $\clase^1$ in the direction of the flow and thus, $\k^C$ is the integral of a H\"older-continuous function along the orbits of $\phi.$
\end{proof}

The following lemma is useful. 

\begin{lema}[{\cite[Section 3]{quantitative}}]\label{lema:positiva}Consider a H\"older-continuous function $f:X\to\R,$ such that $$\frac1{p(\tau)}\int_{\tau} f>k,$$ for some positive $k$ and every periodic orbit $\tau$ of $\phi.$ Then $f$ is Liv\v sic-cohomologous to a positive H\"older-continuous function.
\end{lema}

In order to study the ergodic theory of Anosov flows, Bowen \cite{bowen2} and Ratner \cite{ratner2} introduced the notion of Markov coding.

\begin{defi}\label{defi:markovpatition}The triple $(\E,\pi,r)$ is a \emph{Markov coding} for $\phi,$ if $\E$ is an irreducible two-sided subshit of finite type, the maps $\pi:\E\to X$ and $r:\E\to\R_+^*$ are H\"older-continuous and verify the following conditions: Let $\sigma:\E\to\E$ be the shift, and let $\hat r:\E\times\R\to\E\times\R$ be the homeomorphism defined by $$\hat r(x,t)=(\sigma x, t-r(x)),$$ then

\begin{itemize}\item[i)] the map $\Pi:\E\times\R\to X$ defined by $\Pi(x,t)=\phi_t(\pi( x))$ is surjective and $\hat r$-invariant, \item[ii)] consider the suspension flow $\sigma^r=(\sigma^r_t)_{t\in\R}$ on $(\E\times\R)/\hat r,$ then the induced map $\Pi:(\E\times\R)/\hat r\to X$ is bounded-to-one and, injective on a residual set which is of full measure for every ergodic invariant measure of total support of $\sigma^r.$ 
\end{itemize}
\end{defi}

\begin{obs}\label{obs:weak} If a flow $\phi$ admits a Markov coding then every reparametrization $\psi$ of $\phi$ also admits a Markov coding, simply by changing the roof function $r.$
\end{obs}

A Markov coding is a very accurate measurable model for a flow $\phi.$ If $\phi$ admits a Markov coding, then it has a unique probability measure of maximal entropy, and the function $\Pi:(\E\times\R)/\hat r\to X$ is an isomorphism, between the probability measures of maximal entropy of $\sigma^r$ and that of $\phi.$ In particular the topological entropy of $\phi$ coincides with that of $\sigma^r.$

Recall that a flow $\phi$ is \emph{transitive} if it has a dense orbit.

\begin{teo}[Bowen \cite{bowen1,bowen2}] A transitive Anosov flow admits a Markov coding.
\end{teo}

The following is standard.

\begin{prop}[Bowen-Ruelle \cite{bowenruelle}]\label{prop:ruellebowen} Let $\phi$ be a transitive Anosov flow. Then, given a H\"older-continuous function $f:X\to\R,$ there exists a unique equilibrium state for $f,$ moreover, the equilibrium state is ergodic.
\end{prop}

The equilibrium state of the last proposition can be described as follows (see Bowen-Ruelle \cite[Proposition 3.1]{bowenruelle}). If $(\E,\pi,r)$ is a Markov coding for the Anosov flow $\phi,$ then consider the function $F:\E\to\R$ defined by $$F(x)=\int_0^{r(x)}f(\phi_t(\pi x))dt,$$ and consider the equilibrium state $\nu,$ of $F-P(f)r,$ then the for every measurable function $G:X\to\R$ one has \begin{equation}\label{eq:equlibrio} \int_X Gdm_f=\frac1{\int rd\nu}\int_\E\int_0^{r(x)}G(\phi_t(\pi x ))dtd\nu(x).\end{equation}

We finish this subsection with the following classical result.

\begin{teo} Let $M$ be a closed connected, negatively curved Riemannian manifold. Then the geodesic flow of $M$ on $T^1M,$ is a transitive Anosov flow.
\end{teo}

%We will use the following Corollary of \cite{quantitative}:

%\begin{cor}[{\cite[Section 2]{quantitative}}] Let $\G$ be a cocompact group of isometries of a complete simply connected manifold of negative curvature $\w M.$ Let $\psi$ be a H\"older reparametrization of the geodesic flow $\phi.$ Then the group generated by the periods of $\psi$ is dense in $\R.$
%\end{cor}

%\begin{lema}\label{lema:markovpartition} Let $(\E,\pi,r)$ be a Markov coding for a transitive Anosov flow $\phi_t:X\mismo.$ Set $\psi_t:X\mismo$ to be a H\"older reparametrization of $\phi_t$ by $F:X\to\R_+^*$ and define $f:\E\to\R_+^*$ as $$f(z)=\int_0^{r(z)}F\phi_s(\pi(z))ds.$$ Then $(\E,\pi,f)$ is a Markov coding for $\psi_t.$ If moreover $\phi_t$ does not admit a cross section then the translation flow $\sigma^f_t:\E\times\R/\hat f\mismo$ is topologically weakly mixing.
%\end{lema}

\subsubsection*{H\"older cocycles on $\bord\G$}

Let $M$ be a closed connected negatively curved Riemannian manifold $M,$ and denote by $\w M\to M$ its universal cover. The group $\G=\pi_1(M)$ is hyperbolic, and the visual boundary of $\w M$ is identified with the boundary at infinity $\bord\G$ of the group, endowed with its usual H\"older structure (see Ghys-delaHarpe \cite{ghysharpe}). We will now focus on H\"older cocycles on $\bord\G.$

\begin{defi}\label{defi:cociclo}A \emph{H\"older cocycle} is a map
$c:\G\times\bord\G\to V,$ such that $$c(\g_0\g_1,x)=c(\g_0,\g_1x)+c(\g_1,x),$$ for
any $\g_0,\g_1\in\G$ and $x\in\bord\G,$ and such that $c(\g,\cdot)$ is H\"older-continuous, for every $\g\in\G$ (the same exponent is assumed for every $\g\in\G$). 
\end{defi}

Recall that each $\g\in\G-\{e\},$ has two fixed points on $\bord\G,$ $\g_+$ and $\g_-,$ and that for every $x\in\bord\G-\{\g_-\}$ one has $\g^n x\to \g_+,$ as $n\to\infty.$ We will refer to $\g_+$ as the \emph{attractor} of $\g.$ The \emph{period} of $\g$ for a H\"older cocycle $c,$ is defined by $$\l_c(\g)=c(\g,\g_+).$$ The cocycle property implies that for all $n\in\N,$ one has $\l_c(\g^n)=n\l_c(\g),$ and $\l_c(\g)$ only depends on the conjugacy class $[\g]$ of $\g.$ 

Two cocycles $c$ and $c'$ are \emph{cohomologous}, if there exists a H\"older-continuous function $U:\bord\G\to V,$ such that for all $\g\in\G$ one has $$c(\g,x)-c'(\g,x)=U(\g x)-U(x).$$ One easily deduces from the definition that, the set of periods of a H\"older cocycle is a cohomological invariant. The following theorem of Ledrappier \cite{ledrappier} relates H\"older cocycles with H\"older-continuous maps $T^1 M\to V.$ 

Recall that the periodic orbits of the geodesic flow of $M,$ are in one-to-one correspondence with the conjugacy classes $[\g],$ of elements $\g\in\G-\{e\}.$

\begin{teo}[Ledrappier {\cite[page 105]{ledrappier}}]\label{teo:ledrappier} For each H\"older cocycle $c:\G\times\bord\G\to V,$ there exists a H\"older-continuous map $F_c:T^1M\to V,$ such that for every $\g\in\G-\{e\},$ one has $$\l_c(\g)=\int_{[\g]} F_c.$$ The map $c\mapsto F_c$ induces a bijection between the set of cohomology classes of $V$-valued H\"older cocycles, and the set of Liv\v sic-cohomology classes, of H\"older-continuous maps from $T^1M\to V.$ 
\end{teo}

Two H\"older cocycles $c$ and $\vo c$ are \emph{dual cocycles}, if for every $\g\in\G-\{e\},$ one has $\l_{\vo c}(\g)=\l_c(\g^{-1}).$ If this is the case we will say that the pair $\{c,\vo c\}$ is a pair of dual cocycles. 

Denote by $\bord^2\G$ the set of pairs $(x,y)\in\bord\G^2,$ such that $x\neq y.$ A function $$[\cdot,\cdot]:\bord^2\G\to V$$ is a \emph{Gromov product} for a pair of dual cocycles $\{c,\vo c\},$ if for every $\g\in\G$ and $(x,y)\in\bord^2\G$ one has $$[\g x,\g y]-[x,y]=-(\vo c(\g,x)+c(\g,y)).$$

\begin{obs} The existence of these objects, for a given H\"older cocycle, is a consequence of Ledrappier's Theorem \ref{teo:ledrappier}, see \cite[Section 2]{quantitative} for details.
\end{obs}

We will now focus on real valued H\"older cocycles with non negative periods, i.e. such that $\l_c(\g)\geq0$ for every $\g\in\G-\{e\}.$ The \emph{exponential growth rate} of such cocycle is defined by $$h_c= \limsup_{s\to\infty} \frac{\log\#\{[\g]\in[\G]-\{e\}:\l_c(\g)\leq s\}}s\in(0,\infty],$$ (it is a consequence of Ledrappier's work \cite{ledrappier}, that a H\"older cocycle $c$ with non negative periods, verifies $h_c>0$). 

\begin{obs} A simple argument shows that two dual cocycles have the same exponential growth rate, i.e. $h_c=h_{\vo c}.$
\end{obs}

For $\g\in\G-\{e\},$ denote by $|\g|$ the length of the closed geodesic on $M$ associated to $[\g].$

We will need the following two lemmas.

\begin{lema}[Ledrappier {\cite[page 106]{ledrappier}}]\label{lema:piedra} Let $c$ be a H\"older cocycle with nonnegative periods and finite exponential growth rate, then $$\frac 1m<\inf_{\g\in\G-\{e\}} \frac{\l_c(\g)}{|\g|}<\sup_{\g\in\G-\{e\}}\frac{\l_c(\g)}{|\g|}<m,$$ for a positive $m.$
\end{lema}

\begin{lema}[{\cite[Section 2]{quantitative}}]\label{lema:funcionpositiva} Let $c:\G\times\bord\G\to\R$ be a H\"older cocycle with nonnegative periods and finite exponential growth rate, then the function $F_c$ is Li\v vsic-cohomologous to a positive function.
\end{lema}

If $c$ has finite exponential growth rate then, following Patterson's construction, Ledrappier \cite{ledrappier} proves the existence of a \emph{Patterson-Sullivan} probability measure $\mu$ on $\bord\G$ of cocycle $h_cc,$ this is to say, $\mu$ verifies $$\frac{d\g_*\mu}{d\mu}(x)=e^{-h_cc(\g^{-1},x)}$$ for every $\g\in\G$ and $x\in\bord\G.$ 

\begin{teo}[Ledrappier \cite{ledrappier} page 102]\label{teo:medidas} Let $c$ be a H\"older cocycle with nonnegative periods. Then $c$ has finite exponential growth rate $h_c$ if and only if there exists a Patterson-Sullivan probability measure of cocycle $h_cc.$ If this is the case, the Patterson-Sullivan probability measure is unique.
\end{teo}

Denote by $\mu$ and $\vo\mu$ the Patterson-Sullivan probability measures associated to $c$ and $\vo c$ respectively, and consider a Gromov product $[\cdot,\cdot],$ for the pair $\{c,\vo c\}.$ Remark that the measure $$e^{-h_c[x,y]}d\vo\mu(x)d\mu(y)$$ on $\bord^2\G,$ denoted from now on by $e^{-h_c[\cdot,\cdot]}\vo\mu\otimes\mu,$ is $\G$-invariant. The following theorem is crucial to understand the Weyl chamber flow.

\begin{teo}[The Reparametrizing Theorem \cite{quantitative}]\label{teo:reparametrisation} Let $c$ be a H\"older cocycle with nonnegative periods such that $h_c$ is finite. Then: \begin{enumerate}\item the action of $\G$ in $\bord^2\G\times\R$ via $c,$ that is, $$\g(x,y,s)=(\g x,\g y,s-c(\g,y)),$$ is proper and cocompact. Moreover, the flow $\psi$ on $\G\/(\bord^2\G\times\R),$ defined by $$\psi_t\G(x,y,s)=\G(x,y,s-t),$$ is conjugated to a H\"older reparametrization of the geodesic flow on $T^1 M.$ The conjugating map is also H\"older-continuous. The topological entropy of $\psi$ is $h_c.$ \item The measure $$e^{-h_c[\cdot,\cdot]}\vo\mu\otimes\mu\otimes ds$$ on $\bord^2\G\times\R,$ induces on the quotient $\G\/(\bord^2\G\times\R),$ a positive multiple of the probability measure of maximal entropy of $\psi.$
\end{enumerate}
\end{teo}

\begin{obs}\label{obs:F} Consider $F_c:T^1M\to\R$ given by Ledrappier's Theorem \ref{teo:ledrappier} for the cocycle $c.$ Lemma \ref{lema:funcionpositiva} implies that $F_c$ is Liv\v sic-cohomologous to a positive function. The reparametrization on Theorem \ref{teo:reparametrisation} is given by this positive function. 
\end{obs}

\section{The action by translations of $V$ on $\G\/(\bord^2\G\times V)$}\label{section:mixingweyl}

Recall that $M$ is a closed connected, negatively curved Riemannian manifold, $\G$ is its fundamental group (for any base point), and $V$ is a finite dimensional vector space.

Fix a H\"older cocycle $c:\G\times\bord\G\to V,$ and denote by $\cone_c,$ the smallest closed convex cone of $V,$ that contains the periods $\{\l_c(\g):\g\in\G-\{e\}\}.$ The \emph{dual cone} of $\cone_c$ is the set of linear forms that are nonnegative on this cone: $$\cone_c^*=\{\varphi\in V^*:\varphi|_{\cone_c}\geq0\}.$$ For $\varphi\in \cone_c^*,$ denote by $c_\varphi$ the real valued H\"older cocycle $\varphi\circ c,$ and by $h_\varphi$ the exponential growth rate of $c_\varphi,$ $$h_\varphi=\limsup_s\frac{\log \#\{[\g]:\varphi(\l_c(\g))\leq s\}}s.$$

A direct consequence of the Reparametrizing Theorem \ref{teo:reparametrisation} is the following.
\begin{samepage}
\begin{cor}\label{cor:propia} If there exists $\varphi\in \cone_c^*$ such that $h_\varphi$ is finite, then the action of $\G$ on $\bord^2\G\times V$ via $c,$ that is, $$\g(x,y,v)=(\g x, \g y,v-c(\g, y)),$$ is properly discontinuous. 
\end{cor}
\begin{flushright}$\square$\end{flushright}
\end{samepage}
%\begin{proof} 
%One applies Theorem \ref{teo:reparametrisation} to the cocycle $c_\varphi.$ 
%\end{proof}

Denote by $\inte(\cone^*_c)$ the interior of $\cone^*_c.$ One has the following lemma.

\begin{lema} If $\varphi\in\cone_c^*$ is such that $h_\varphi<\infty,$ then $\varphi\in\inte(\cone_c^*),$ in particular $\inte(\cone^*_c)$ is nonempty. Moreover, for every $\theta\in\inte(\cone^*_c),$ one has $h_\theta<\infty.$
\end{lema}

\begin{proof}Consider the function $F_c:T^1 M\to V$ associated to $c.$ One has $$ \varphi(\int_{[\g]} F_c)= \varphi(\l_c(\g))\geq0.$$ Moreover, since $h_\varphi<\infty,$ Ledrappier's Lemma \ref{lema:piedra}, applied to the H\"older cocycle $c_\varphi,$ implies that $$\varphi(\frac 1{|\g|}\int_{[\g]}F_c)=\frac 1{|\g|}\varphi(\l_c(\g))>k>0,$$ for some positive $k$ and every $\g\in\G-\{e\}.$ Anosov's closing Lemma (c.f. Shub \cite{shub}) states that the convex combinations of the Lebesgue measures on periodic orbits, are dense in $\cal M^\phi,$ thus 
\begin{itemize}
 \item[-] $\varphi(\int F_cdm)>k$ for every $\phi$-invariant probability measure $m,$ \item[-] the set $$\{\int F_c dm:m\in\cal M^\phi\}$$ is compact and generates the cone $\cone_c.$
\end{itemize}
Hence, $\varphi$ is positive on the cone $\cone_c-\{0\},$ i.e. $\varphi\in\inte(\cone_c^*).$ 

If $\theta$ belongs to the interior of $\cone^*_c,$ then $\theta|_{\cone_c-\{0\}}>0.$ Hence, there exists a positive $a$ such that $\varphi(v)\leq a\theta(v),$ for all $v\in\cone_c.$ This implies that $h_\theta\leq a h_\varphi<\infty.$ This finishes the proof.
\end{proof}

Assume from now on the existence of $\varphi\in \cone_c^*$ with finite $h_\varphi.$ We then have a natural map between $\P(\inte(\cone^*_c))$ and $\P(\cone_c)$ as follows. Fix $F_c:T^1 M\to V$ associated to $c.$

\begin{defi} For $\varphi\in \inte(\cone_c^*),$ denote by $m_\varphi$ the equilibrium state, on $T^1 M,$ of the function $-h_\varphi\varphi\circ F_c$ (recall Proposition \ref{prop:ruellebowen}). The \emph{dual direction} of $\R_+\varphi,$ is the direction in $\cone_c$ given by the vector $$\int F_cdm_\varphi,$$ and is denoted by $\uu\varphi\in\P(\cone_c).$ 
\end{defi}

\begin{obs} A change in the Liv\v sic-cohomology class of $F_c$ does not change the value of the integral of $F_c$ over any $\phi$-invariant measure. Hence $\uu\varphi$ is well defined, independently of the choice of $F_c.$ Remark also that if $t\in\R_+,$ then $h_{t\varphi}=h_\varphi/t,$ hence, the dual direction of $\R_+\varphi,$ only depends on the direction given by $\varphi.$
\end{obs}

Fix also a dual cocycle $\vo c$ of $c,$ and a Gromov product $[\cdot,\cdot]:\bord^2\G\to V$ for the pair $\{c,\vo c\}.$ Denote by $\mu_\varphi$ and $\vo\mu_\varphi,$ the Patterson-Sullivan probability measures of cocycles $h_\varphi c_\varphi$ and $h_\varphi\vo c_\varphi$ respectively. The function $$[\cdot,\cdot]_\varphi= \varphi \circ[ \cdot,\cdot]$$ is a Gromov product for the pair $\{c_\varphi, \vo c_\varphi\}.$ Denote by $\Pats_\varphi,$ the measure on $\G\/(\bord^2\G\times V),$ induced by the measure $$\widetilde{\Pats_\varphi}=e^{-h_\varphi[\cdot,\cdot]_\varphi} \vo\mu_\varphi \otimes \mu_\varphi\otimes\Leb_V,$$ where $\Leb_V$ is a fixed Lebesgue measure on $V.$ The measure $\Pats_\varphi$ is called the \emph{Bowen-Margulis measure} of the pair $\{c,\vo c\}$ for the linear form $\varphi.$

%Since $h_\varphi$ is finite lemma \ref{lema:positiva} applies and we can assume that $\varphi\circ F:T^1M\to\R$ is strictly positive, consider then $u_\varphi$ with $\varphi(u_\varphi)=1$ in the direction given by the vector  
%On fixe une norme sur $V$ et une constante $L>0$ et on consid\`ere le sous ensemble de $V$ $$P_C:=\{tu_\varphi+w:w\in\ker\varphi\textrm{ avec }\|w\|\leq \sqrt t L\}.$$

Choose a vector $u_\varphi\in\uu\varphi,$ such that $\varphi(u_\varphi)=1,$ and consider the flow $\om^\varphi=(\om^\varphi_t)_{t\in\R}$ on $\G\/(\bord^2\G\times V),$ induced on the quotient by $$(x,y,v)\mapsto(x,y,v-tu_\varphi).$$

\begin{prop}[Straightening the action of $V$]\label{prop:topo} For every $\varphi\in\cone_c^*$ such that $h_\varphi<\infty,$ there exists a H\"older reparametrization of the geodesic flow $\psi=\psi^{c,\varphi},$ a H\"older-conti\-nuous map $f:T^1M\to \ker \varphi,$ with zero mean for the probability measure of maximal entropy of $\psi,$ denoted by $m^\#,$ i.e. $$\int_{T^1 M} fdm^\#=0,$$ and a H\"older-continuous homeomorphism $$\vo E:\G\/(\bord^2\G\times V)\to T^1M\times\ker\varphi,$$ that conjugates the flow $\om^\varphi,$ with the flow $\widehat\psi=(\widehat\psi_t)_{t\in\R}$ on $T^1M\times \ker\varphi,$ defined by \begin{equation}\label{equation:rompebola}\widehat\psi_t(p,v_0)=(\psi_t (p),v_0-\int_0^tf(\psi_sp)ds).\end{equation} The map $\vo E$ also conjugates the actions of $\ker\varphi,$ on $\G\/(\bord^2\G\times V)$ and on $T^1M\times\ker\varphi$ (by translation on the fibers), and is an isomorphism, up to a multiplicative constant, between the measures $\Pats_\varphi$ and $m^\#\otimes\Leb_{\ker\varphi}.$
\end{prop}

\begin{proof} Consider the action of $\G$ on $\bord^2\G\times\R$ via $c_\varphi.$ Then one has a $\G$-equivariant fibration $\widehat\varphi:\bord^2\G\times V\to \bord^2\G\times\R$ with fiber $\ker\varphi,$ given by $$\widehat\varphi (x,y,v)=(x,y,\varphi(v)).$$ The measure $\w{\Pats_\varphi}$ disintegrates over the measure $$e^{-h_\varphi[\cdot, \cdot]_\varphi} \vo\mu_\varphi \otimes \mu_\varphi \otimes\Leb_\R$$ on $\bord^2\G\times\R,$ with conditional measures the Lebesgue measure on $\ker\varphi.$

Since $h_\varphi$ is finite, the Reparametrizing Theorem \ref{teo:reparametrisation} applies and thus, the action of $\G$ on $\bord^2\G\times\R$ via $c_\varphi,$ is properly discontinuous. Moreover there exists a H\"older-continuous homeomorphism $E:\G\/(\bord^2\G\times\R)\to T^1M,$ that conjugates the translation flow with a reparametrization of the geodesic flow. Denote this reparametrization by $\psi.$ The image of the measure induced on the quotient by $$e^{-h_\varphi[\cdot,\cdot]_\varphi}\vo \mu_\varphi \otimes \mu_\varphi \otimes \Leb_\R,$$ is sent by $E,$ to a positive multiple of the (unique) probability measure of maximal entropy of $\psi.$ 

The functions $\varphi\circ F_c$ and $F_{c_\varphi}$ are Liv\v sic-cohomologous, since they have the same period, for every periodic orbit of the geodesic flow.  Lemma \ref{lema:funcionpositiva} implies then that, $\varphi\circ F_c$ is Liv\v sic-cohomolgous to a positive function, hence we can (and will) assume that $\varphi\circ F_c>0.$ Remark \ref{obs:F} states that the flow $\psi$ can be taken as the re\-pa\-ra\-me\-tri\-za\-tion of the geodesic flow $\phi$ by $\varphi\circ F_c.$ The probability measure of maximal entropy of $\psi$ is ${m_\varphi}^\#$ (recall that $m_\varphi$ is the equilibrium state of $-h_\varphi\varphi\circ F_c$ and use Lemma \ref{lema:entropia2}).

Abusing notation, denote by $$\widehat\varphi:\G\/(\bord^2\G\times V)\to \G\/(\bord^2\G\times\R),$$ the map induced on the quotients by $\widehat\varphi:\bord^2\G\times V\to \bord^2\G\times\R.$ For every $u\in V,$ one has that $$E\circ\widehat\varphi(x,y,v-u)=\psi_{\varphi(u)}(E(x,y,\varphi(v))),$$ in particular the flow $\om^\varphi$ is (semi)conjugated to $\psi$ by $E\circ\widehat\varphi,$ i.e. for every $t\in\R$ one has $$E\circ\widehat\varphi\circ\om^\varphi_t=\psi_t\circ E\circ\widehat\varphi.$$

The action of the abelian group $\ker \varphi,$ on $\bord^2\G\times V,$ commutes with the action of $\G$ and preserves the fibers $\widehat\varphi^{-1}(x,y,t)$ of $\widehat\varphi.$ Hence we have an action of $\ker\varphi$ on the quotient, and one finds that $$E\circ\widehat\varphi:\G\/(\bord^2\G\times V)\to T^1M$$ is a vector bundle with fiber $\ker \varphi,$ and the group $\ker\varphi$ acts by H\"older-continuous homeomorphisms on $\G\/(\bord^2\G\times V)$ preserving the fibers, and acting transitively on them. Using the zero section of a vector bundle, and the action of $\ker \varphi,$ one can trivialize this bundle. Hence, $\G\/(\bord^2\G\times V)$ is (H\"older) isomorphic to $T^1M\times\ker\varphi,$ and this isomorphism is $\ker\varphi$-equivariant.

Denote by $\Psi=(\Psi_t)_{t\in\R}$ the flow on $T^1M\times\ker\varphi,$ corresponding to the flow $\om^\varphi$ via this last identification. Since $\om^\varphi$ commutes with the action of $\ker\varphi,$ the same occurs for $\Psi,$ and thus we can write $$\Psi_t(p,v_0)=(\psi_t(p),v_0-\k(p,t)),$$ where $\k:T^1M\times \R\to\ker\varphi$ is a translation cocycle over $\psi.$ Lemma \ref{lema:livsicgral} implies the existence of a H\"older-continuous map $f:T^1M\to\ker \varphi,$ such that the cocycles $\k$ and $\k_f$ are Liv\v sic-cohomologous (for the flow $\psi$). The flow $\Psi$ is hence conjugated to the flow $\widehat\psi=(\widehat\psi_t)_{t\in\R}$ on $T^1M\times\ker \varphi,$ defined by $$\widehat{\psi_t}(p,v)=(\psi_t(p),v-\int_0^tf(\psi_s(p))ds).$$ Denote by $\vo E:\G\/(\bord^2\G\times V)\to T^1M\times\ker\varphi,$ the composition of the trivialization of $\G\/(\bord^2\G\times V)$ defined above, with this last conjugacy between $\Psi$ and $\widehat\psi.$ By definition, $\vo E$ conjugates the flows $\om^\varphi$ and $\widehat\psi,$ and is $\ker\varphi$-equivariant.

We remark that the image by $\vo E,$ of the measure $\Pats_\varphi$ on $T^1M\times \ker\varphi,$ is a measure that disintegrates as a $\ker\varphi$-invariant measure on the fibers, and a positive constant multiple of ${m_\varphi}^\#$ on $T^1M.$ This measure is then a positive constant multiple of ${m_\varphi}^\# \otimes \Leb_{\ker \varphi}.$ 

It remains to check that $\int_{T^1M}fd{m_\varphi}^\#=0.$ In order to do this, recall that $\varphi(u_\varphi)=1$ and that $u_\varphi$ is collinear to the vector $\int F_cdm_\varphi,$ hence \begin{equation}\label{eq:1}\int F_cdm_\varphi=u_\varphi\int\varphi\circ F_cdm_\varphi.\end{equation}

\begin{figure}[h]\label{figure:figura1}
\begin{center}
\psfrag{a}{$\w\psi_{\varphi(\l_c(\g))}(\cdot)$}\psfrag{b}{$\l_c^0(\g)$}\psfrag{c}{$\textrm{fiber over $p$}$}
\includegraphics[height=7cm]{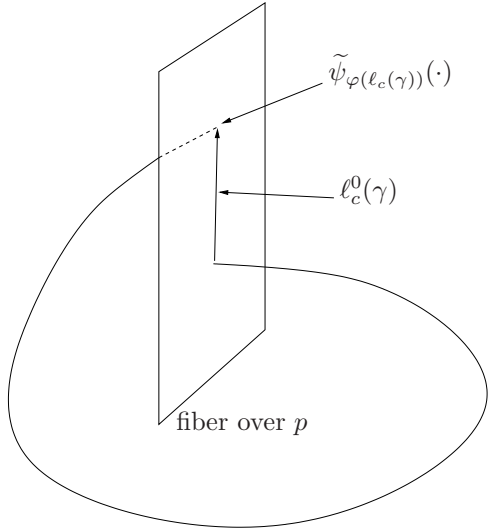}
\caption{If $p\in T^1M$ belongs to the periodic orbit associated to $[\g],$ the translation on the fiber $\ker\varphi$ by the flow $\widehat\psi,$ at the returning time, is given by $\l_c^0(\g).$} 
\end{center}
\end{figure}

For every $\g\in\G-\{e\},$ let $\l_c^0(\g)$ be the projection of the period $\l_c(\g)$ on $\ker\varphi,$ using the decomposition $V=\ker\varphi\oplus\uu\varphi.$ Remark that, for any $v\in V$ and $\g\in\G-\{e\},$ one has $$\g(\g_-,\g_+,v+\l^0_c(\g))=(\g_-,\g_+,v-\l_{c_\varphi}(\g)u_\varphi)=\om^\varphi_{\l_{c_\varphi}(\g)}(\g_-,\g_+,v).$$ This is to say, $\l^0_c(\g)$ is the displacement on $\ker\varphi$ of the flow $\om^\varphi,$ over a point of the form $(\g_-,\g_+,v),$ on the return time $\varphi(\l_c(\g))=\l_{c_\varphi}(\g).$

Consider also $F_c=F^0_c+(\varphi\circ F_c )u_\varphi,$ using this same decomposition. Equation (\ref{eq:1}) implies that $$\int_{T^1M} F^0_cdm_\varphi=0,$$ and that $$\l_c^0(\g)=\int_{[\g]} F^0_c(\phi_sp) ds,$$ where $\phi$ is the geodesic flow of $M.$ Since $\widehat\psi$ and $\om^\varphi$ are conjugated, one has $$\l^0_c(\g)=-\int_{[\g]^\#} f=-\int_{[\g]} f\varphi\circ F_c$$ (see the above figure), where $[\g]^\#$ is the $\psi$-invariant measure associated to the periodic orbit $[\g],$ defined by equation (\ref{equation:numeral}). Liv\v sic's Theorem \ref{teo:livsic3} implies that, the functions $F^0_c$ and $-f\varphi\circ F_c$ are Liv\v sic-cohomologous for the flow $\phi,$ thus $$0=\int F^0_cdm_\varphi=-\int f\varphi\circ F_cdm_\varphi=-\int fd{m_\varphi}^\#\int\varphi\circ F_cdm_{\varphi}.$$ This finishes the proof.
\end{proof}

\subsubsection*{Mixing properties of the action of $V$ on $\G\/(\bord^2\G\times V)$}

Now that we have a good description of $\G\/(\bord^2\G\times V),$ together with the action of $V,$ we can use Markov codings and a theorem of Thirion \cite{thirionmixing}, to prove a mixing property.

Consider $\varphi\in \cone_c^*,$ with $h_\varphi<\infty,$ and $u_\varphi\in\uu\varphi$ such that $\varphi(u_\varphi)=1.$ For $v_0\in\ker\varphi$ and $t\in\R^+,$ denote by $\om_t^{\varphi,v_0}:\G\/(\bord^2\G\times V)\to \G\/(\bord^2\G\times V),$ the map induced on the quotient by $$(x,y,v)\mapsto(x,y,v-tu_\varphi-\sqrt t v_0).$$ 

If $|\cdot|$ is an euclidean norm on $V,$ denote by $I=I^{|\cdot|}:\ker\varphi\to\R$ the function defined by \begin{equation}\label{eq:I}I(v)=\frac{|v|^2|u_\varphi|^2-\<v,u_\varphi\>^2}{|u_\varphi|^2}.\end{equation}

\begin{teo}\label{teo:melange} Let $c:\G\times\bord\G\to V$ be a H\"older cocycle, such that the group generated by its periods is dense in $V.$ Fix a linear form $\varphi\in \cone_c^*$ such that $h_\varphi<\infty.$ There there exists an euclidean norm $|\cdot|$ on $V,$ such that the map $\om_t^{\varphi,v_0}$ verifies: given two compactly supported continuous functions $f_0,f_1:\G\/(\bord^2\G\times V)\to\R,$ one has $$(2\pi t)^{(\dim V-1)/2}\Pats_\varphi( f_0\cdot f_1\circ\om_t^{\varphi,v_0})\rightarrow c e^{-\II(v_0)/2}\Pats_\varphi(f_0)\Pats_\varphi(f_1),$$ as $t\to\infty,$ for some positive constant $c.$ 
\end{teo}

The remainder of the section is devoted to the proof of Theorem \ref{teo:melange}.

Applying Proposition \ref{prop:topo}, we get a reparametrization of the geodesic flow $\psi,$ together with a H\"older-continuous map $f:T^1M\to \ker\varphi$ and $\vo E:\G\/(\bord^2\G\times V) \to T^1M\times\ker\varphi$ that conjugates:

\begin{itemize} \item[-] the action of $\ker\varphi$ on $\G\/(\bord^2\G\times V),$ and it's action by translations on the fibers on $T^1M\times\ker\varphi,$ \item[-] the flow $\om^\varphi$ on $\G\/(\bord^2\G\times V)$ with the flow $\widehat\psi=(\widehat\psi_t)_{t\in\R}$ on $T^1M\times\ker\varphi,$ defined by equation (\ref{equation:rompebola}).
\end{itemize}

We will thus study mixing properties of $$t\cdot(x,v) \mapsto (\psi_t(p),v-\int_0^tf(\psi_sp)ds-\sqrt tv_0).$$

Consider a Markov coding $(\E,\pi,r)$ for $\psi$ (Remark \ref{obs:weak}). According to equation (\ref{eq:equlibrio}), there exists an equilibrium state of the shift $\sigma:\E\to\E,$ denoted by $\nu_\varphi,$ corresponding to the measure ${m_\varphi}^\#$ via the Markov coding, i.e. for every measurable function $G:T^1M\to\R$ one has \begin{equation}\label{equation:cuenta} \int_{T^1M} Gd{m_\varphi}^\#=\frac1{\int rd\nu_\varphi}\int_\E\int_0^{r(x)}  G(\psi_s(\pi (x)))ds d\nu_\varphi(x).\end{equation}

Define $K:\E\to V$ by $$K(x)=r(x)u_\varphi +\int_0^{r(x)} f(\psi_s(\pi (x)))ds,$$ and $\hat K:\E\times V\to\E\times V$ by $\hat K(x,v)=(\sigma x, v-K(x)).$

\begin{lema} The map $\sub \pi:\E\times V\to T^1M\times \ker\varphi,$ defined by $$\sub \pi(x,v)= (\psi_{\varphi(v)}(\pi x),v-\varphi(v)u_\varphi-\int_0^{\varphi(v)}f(\psi_s(\pi (x)))ds)$$ $$=\widehat\psi_{\varphi(v)}(\pi x, v-\varphi(v)u_\varphi),$$ is $\hat K$-invariant, and induces a measurable isomorphism, between the measure induced on $(\E\times V)/\hat K$ by $\nu_\varphi\otimes\Leb_V,$ and a positive multiple of the measure ${m_\varphi}^\#\otimes\Leb_{\ker\varphi}$ on $T^1M\times\ker\varphi.$
\end{lema}

\begin{proof} Let's show that $\sub\pi$ is $\hat K$-invariant, the proof is an explicit computation. Remark that property i) on the definition of Markov coding states that, for every $x\in\E$ and $t\in\R,$ one has $\psi_{t-r(x)}(\pi(\sigma(x)))=\psi_t(\pi(x)).$ Now, $$\sub\pi(\hat K(x,v))=\widehat\psi_{\varphi(v)-r(x)}(\pi(\sigma(x)),v-K(x)-\varphi(v-K(x))u_\varphi).$$ Recall that $K(x)=r(x)u_\varphi+ \int_0^{r(x)}f(\psi_s(\pi x))ds,$ hence $$\sub\pi(\hat K(x,v))=\widehat\psi_{\varphi(v)-r(x)}(\pi(\sigma(x)),v-\int_0^{r(x)}f(\psi_s(\pi(x)))ds-\varphi(v)u_\varphi)=$$ $$ \widehat\psi_{\varphi(v)}(\pi(x), v-\varphi(v)u_\varphi -\int_0^{r(x)}f(\psi_s(\pi x))ds-\int_0^{-r(x)}f(\psi_s(\pi(\sigma(x))))ds.$$ Finally, remark that $$-\int_0^{-r(x)}f(\psi_s(\pi(\sigma(x))))ds=-\int_{r(x)}^0 f(\psi_{s-r(x)}(\pi(\sigma(x))))ds=$$ $$-\int_{r(x)}^0f(\psi_s(\pi(x)))ds=\int_0^{r(x)}f(\psi_s(\pi(\sigma(x))))ds.$$ This proves $\hat K$-invariance. The remaining statements follow from equation (\ref{equation:cuenta}) and property ii) of Markov codings.
\end{proof}

Hence, the flow $\widehat\psi$ is measurably conjugated to the translation flow on $(\E\times V)/\hat K,$ in the direction given by $u_\varphi.$ Remark that, since Proposition \ref{prop:topo} states that $\int fdm^\#_\varphi=0,$ equation (\ref{equation:cuenta}) applied to $G=f$ yields $$\int_{\E}Kd\nu_\varphi=(u_\varphi+\int fdm_\varphi^\#)\int_\E rd\nu_\varphi=u_\varphi\int_\E rd\nu_\varphi.$$ Moreover, this conjugation also conjugates the action of $\ker\varphi$ on $T^1M\times\ker\varphi,$ and on $(\E\times V)/\hat K.$ Theorem \ref{teo:melange} is thus a consequence of Proposition \ref{prop:topo}, and the following theorem due to Thi\-rion \cite{thirionmixing}.

\begin{teo}[Thirion \cite{thirionmixing}] Let $\E$ be a subshift of finite type and $K:\E\to V$ a H\"older-continuous map, 
such that the group generated by its periods is dense in $V.$ Assume there exists $\varphi\in V^*$ such that $\varphi\circ K$ is Liv\v sic-cohomologous to a positive function. Consider an equilibrium state $\nu$ and denote by $$\tau=\int_\E K d\nu\in V.$$ Define $\hat K:\E\times V\to\E\times V$ by $\hat K(x,v)=(\sigma(x),v-K(x)).$ Then there exists an euclidean norm $|\cdot|$ on $V$ such that given two compactly supported continuous functions $f_0,f_1:(\E\times V)/\hat K\to\R,$ and $v_0\in \ker\varphi,$ one has $$(2\pi t)^{(\dim V-1)/2}\int_{(\E\times V)/\hat K}f_0(x,v)f_1(x,v-t\tau-\sqrt t v_0)d(\nu\otimes\Leb_V)$$ converges, as $t\to\infty,$ to $$\ca e^{-\II(v_0)/2}\int_{(\E\times V)/\hat K}f_0d(\nu \otimes\Leb_V) \int_{(\E\times V)/\hat K} f_1 d(\nu \otimes\Leb_V),$$ where $\ca>0$ is a constant and $\II(v_0)=(|v_0|^2|\tau|^2-\<v_0,\tau\>^2)/|\tau|^2.$ 
\end{teo}

\begin{proof}Let us give some hints on the proof for completeness, the basic method is that of Guivarc'h-Hardy \cite{guivarchhardy}. Consider a H\"older-continuous function $g:\E_A\to\R,$ and the associated Ruelle operator, defined by $$L_g(f)(x)=\sum_{y\in\E:\sigma(y)=x}e^{-g(y)}f(y),$$ where $f:\E\to\R$ is H\"older-continuous. We can assume that $g$ is normalized, such that the equilibrium state $\nu,$ is the unique probability measure on $\E$ such that $L_g^*\nu=\nu.$ One then considers the semi-Markovian chain on $\E\times V,$ defined by $$\P_{(x,v)}=\sum_{y\in\E:\sigma(y)=x}e^{-g(y)}\delta_{(y,v+K(y))}.$$ The proof then consists on explicitly verifying the hypothesis of Babillot \cite[Theorem 2.9]{babillot}, see Thirion \cite{thirionmixing} for details.
\end{proof}

\section[The Weyl chamber flow]{Convex representations and the Weyl chamber flow}\label{section:convex}

We are now interested in studying representations $\G\to G,$ of the fundamental group $\G$ of a closed connected negatively curved Riemannian manifold, admitting equivariant maps from $\bord\G$ to some flag space of a noncompact real algebraic semisimple group $G.$ 

Let $K$ be a maximal compact subgroup of $G,$ and consider $\tau,$ the Cartan involution on $\frak g=\lie(G)$ whose fixed point set is the Lie algebra of $K.$ Consider $\frak p=\{v\in\frak g: \tau v=-v\}$ and $\frak a$ a maximal abelian subspace contained in $\frak p.$

Let $\E$ be the set of roots of $\frak a$ on $\frak g.$ Consider a closed Weyl chamber $\frak a^+,$ $\E^+$ the set of positive roots associated to $\frak a^+,$ and $\Pi$ the set of simple roots determined by $\E^+.$ Let $W$ be the Weyl group of $\E,$ and denote by $u_0:\frak a\to \frak a$ the longest element in $W,$ which is the unique element in $W$ that sends $\frak a^+$ to $-\frak a^+.$  The \emph{opposition involution} $\ii:\frak a\to\frak a$ is defined by $\ii=-u_0.$

To each subset $\t$ of $\Pi,$ one associates two opposite parabolic subgroups, $P_\t$ and $\wk{P_\t},$ of $G,$ whose Lie algebras are, by definition, $$\frak p_\t=\frak a \oplus\bigoplus_{\a\in\E^+}\frak g_\a\oplus \bigoplus_{\a\in \<\Pi-\t\>}\frak g_{-\a},$$ and $$\wk{\frak p_\t}=\frak a \oplus\bigoplus_{\a\in\E^+}\frak g_{-\a}\oplus \bigoplus_{\a\in \<\Pi-\t\>}\frak g_{\a},$$ where $\<\t\>$ is the set of positive roots generated by $\t,$ and $$\frak g_\a=\{w\in\frak g:[v,w]=\a(v)w\ \forall v\in\frak a\}.$$ Every pair of opposite parabolic subgroup of $G$ is conjugated to $(P_\t,\wk{P_\t}),$ for a unique $\t,$ and every opposite parabolic subgroup of $P_\t,$ is conjugated to $P_{\ii\t}:$ the parabolic group associated to $$\ii\t=\{\a\circ\ii:\a\in\t\}.$$

Fix from now on a subset of simple roots $\t\subset\Pi,$ and denote by $\scr F_\t=G/P_\t.$ The space $\scr F_{\ii\t}\times\scr F_\t$ has a unique open $G$-orbit, denoted by $\posgen_\t.$

\begin{defi} A representation $\rho:\G\to G$ is $\t$-\emph{convex}, if it admits two H\"older-continuous $\rho$-equivariant maps, $\xi:\bord\G\to\scr F_\t$ and $\eta:\bord\G\to\scr F_{\ii\t},$ such that whenever $x\neq y$ in $\bord\G,$ the pair $(\eta(x),\xi(y))$ belongs to $\posgen_\t.$
\end{defi}

The space $\scr F_\Pi=\scr F$ is the Furstenberg boundary of the symmetric space of $G,$ hence, a $\Pi$-convex representation is called \emph{hyperconvex}.

We recall some definitions from Benoist \cite{limite}. An element $g\in G$ is \emph{proximal on} $\scr F_\t,$ if it has an attracting fixed point on $\scr F_\t.$ This attractor is unique and is denoted by $g^\t_+.$ The element $g$ also has a fixed point $g_-^\t$ on $\scr F_{\ii\t},$ which is the attractor for $g^{-1}$ on $\scr F_{\ii\t}.$ For every $x\in\scr F_\t$ such that $(g_-^\t,x)\in\posgen_\t,$ one has $g^nx\to g_+^\t.$ The point $g_-^\t$ is called the repelling hyperplane of $g.$

\begin{lema}[{\cite[Section 3]{exponential}}]\label{lema:lox} Let $\rho:\G\to G$ be a Zariski-dense $\t$-convex representation. Then for every $\g\in\G-\{e\},$ $\rho(\g)$ is proximal on $\scr F_\t,$ $\xi(\g_+)$ is its attracting fixed point and $\eta(\g_-)$ is the repelling hyperplane.
\end{lema}

The equivariant functions $\xi$ and $\eta$ of the definition are then unique, since attracting points $\g_+$ are dense in $\bord\G.$

%An element $g \in G$ is said to be purely loxodromic if its conjugated to an element on $Me^{\frak a^{++}}$ where $\frak a^{++}$ is the interior of the Weyl chamber.

%The following Lemma implies, in particular, that hyperconvex representations are purely loxodromic. Nevertheless one has the following stronger result. Recall that by definition, the limit cone of a subgroup is closed.

%\begin{lema}[{\cite[Section 3]{exponential}}]\label{lema:interior} The limit cone of a Za\-ris\-ki dense hyperconvex representation is contained in the interior of the Weyl chamber $\frak a^+.$
%\end{lema}

%The existence of enough irreducible representations of $G$ implies that Zariski dense hyperconvex representations are purely loxodromic:

%Consider $\Pi$ the set of simple roots of $\frak a$ on $\frak g$ such that $$\frak a^+=\{v\in\frak a:\a(v)\geq0\textrm{ for all }\a\in\Pi\}$$ and consider $\{\om_\a\}_{\a\in\Pi}$ the set of fundamental weights of $\Pi.$ 

\subsubsection*{Busemann cocycle of $\rho$}

To a $\t$-convex representation $\rho:\G\to G,$ one associates a H\"older cocycle on $\bord\G.$ In order to do so, we need \emph{Busemann's cocycle} of $G,$ introduced by Quint \cite{quint1}. 

The set $\scr F$ is $K$-homogeneous, denote by $M$ the stabilizer of $[P]$ in $K.$ One defines $\bus_\Pi:G\times\scr F\to\frak a$ to verify the following equation $$gk=l\exp(\bus_\Pi(g,kM))n,$$ for every $g\in G$ and $k\in K,$ using Iwasawa's decomposition of $G=K\exp(\frak a)N,$ where $N$ is the unipotent radical of $P.$

In order to obtain a cocycle only depending on the set $\scr F_\t$ (and $G$), one considers $$\frak a_\t=\bigcap_{\a\in\Pi-\t}\ker\a,$$ the Lie algebra of the center of the reductive group $P_\t\cap \vo{P_{\t}},$ where $\vo{P_\t}$ is an opposite parabolic group of $P_\t.$ Consider also $p_\t:\frak a\to\frak a_\t,$ the only projection invariant under the group $W_\t=\{w\in W:w(\frak a_\t)=\frak a_\t\}.$ 

\begin{obs}\label{obs:formula}One easily verifies the following relation: $p_{\ii\t}=\ii\circ p_\t\circ\ii.$
\end{obs}

Quint \cite{quint1} proves the following lemma.

\begin{lema}[{Quint \cite[Lemmas 6.1 and 6.2]{quint1}}] The map $p_\t\circ\bus_\Pi$ factors trough a map $\sigma_\t:G\times\scr F_\t\to \frak a_\t.$ The map $\bus_\t$ verifies the cocycle relation: for every $g,h\in G$ and $x\in\scr F_\t$ one has $$\bus_\t(gh,x)=\bus_\t(g,hx)+\bus_\t(h,x).$$
\end{lema}

The cocycle associated to a $\t$-convex representation $\vect_\t^\rho=\vect_\t:\G\times\bord\G\to\frak a_\t,$ is defined by $$\vect_\t(\g,x)= \bus_\t(\rho(\g), \xi(x)).$$

% and $\vect_{\ii\t}^\rho=\vect_{\ii\t}:\G\times\bord\G\to\frak a_{\ii\t}$ defined by $$\vect_{\ii\t}(\g,x)= \bus_{\ii\t}(\rho\g, \eta(x)),$$

Denote by $\lambda:G\to\frak a^+$ the \emph{Jordan projection}, and define $\lambda_\t:G\to\frak a_\t$ by $\lambda_\t(g)=p_\t(\lambda (g)).$

\begin{lema}\label{lema:alpedo}Let $\rho:\G\to G$ be a Zariski-dense $\t$-convex representation. Then the period of $\vect_\t,$ for $\g\in\G-\{e\},$ is $$\vect_\t(\g,\g_+)=\lambda_\t(\rho(\g)).$$
\end{lema}

\begin{proof} The proof follows from Lemma \ref{lema:lox}. See \cite[Lemma 7.5]{quantitative} for details.
\end{proof}

Remark that a $\t$-convex representation is also (by definition), $\ii\t$-convex. Define then $\vo{\vect_\t}:\G\times\bord\G\to\frak a_\t$ by $\vo{\vect_\t}=\ii\vect_{\ii\t}.$ One has the following.

\begin{lema} The pair $\{\vect_\t,\vo{\vect_\t}\}$ is a pair of dual cocycles.
\end{lema}

\begin{proof} The proof follows from Remark \ref{obs:formula}, together with Lemma \ref{lema:alpedo}, and the fact that $\ii(\lambda (g))=\lambda(g^{-1}),$ for every $g\in G.$
\end{proof}

Consider $\cone_{\vect^\t},$ the closed cone associated to $\vect_\t.$ Since $\cone_{\vect^\t}$ is contained in $p_\t(\frak a^+),$ it does not contain any line, and thus the dual cone $\cone_{\vect^\t}^*$ has non empty interior.

%If $\rho(\G)$ is Zariski dense, then following Benoist\cite{limite}, $\cone_\rho^\t$ is convex and has non empty interior on $\frak a_\t,$ since it is the projection to $\frak a_\t$ of the limit cone $\cone_{\rho(\G)}.$ 

\begin{lema}\label{lema:finite} Let $\rho:\G\to G$ be a Zariski-dense $\t$-convex representation, and consider  $\varphi$ in the interior of the dual cone $\cone_{\vect^\t}^*,$ then the cocycle $\varphi\circ\vect_\t:\G\times\bord\G\to\R$ has finite exponential growth rate.
\end{lema}

\begin{proof}The proof follows exactly as \cite[Lemma 7.7]{quantitative}.
\end{proof}

Applying Corollary \ref{cor:propia} to the cocycle $\vect_\t$ one directly obtains:

\begin{cor}Let $\rho:\G\to G$ be Zariski-dense $\t$-convex representation, then the action of $\G$ on $\bord^2\G\times\frak a_\t$ via $\vect_\t,$ is properly discontinuous.
\end{cor}

Even though we will not use it on this work, we remark that Lemma \ref{lema:finite}, together with \cite[Corollary 4.1]{quantitative}, imply the following counting result:

\begin{cor}Let $\rho:\G\to G$ be a Zariski-dense $\t$-convex representation, and consider $\varphi$ in the interior of $\cone_{\vect^\t}^*.$ Then there exists $h_\varphi>0,$ such that $$h_\varphi t e^{-h_\varphi t}\#\{[\g]\in[\G]\textrm{ primitive}:\varphi(\lambda_\t(\rho\g))\leq t\}\to1,$$ as $t\to\infty.$
\end{cor}

\subsubsection*{Gromov product}\label{productogromov}

%\begin{obs} Let $\L:G\to\PGL(V)$ be a finite dimensional, proximal irreducible representation. Consider $\chi$ the maximal weight of $\L$ and $$\t=\{\a\in\Pi:\chi-\a\textrm{ is a weight of }\L\}.$$ Then there exists two analytic equivariant maps $\xi_\L:\scr F_\t\to\P(V)$ and $\eta_\L:\scr F_{\ii\t}\to\P(V^*)$ such that for every pair $(x,y)\in\posgen_\t$ one has $$\eta_\L(x)|\xi_\L(y)\neq0.$$
%\end{obs}

The purpose of this section, is to define a Gromov product for the pair $\{\vect_\t,\vo{\vect_\t}\}.$ We begin with the following result of Tits \cite{tits} (see also Humphreys \cite[Chapter XI]{LAG}).

\begin{prop}[Tits \cite{tits}]\label{prop:titss} For each $\alpha\in\Pi,$ there exists a finite dimensional proximal irreducible representation $\L_\alpha:G\to\PGL(V_\alpha),$ such that the highest weight $\chi_\alpha$ of $\L_\alpha,$ is an integer multiple of the fundamental weight $\om_\alpha.$ Moreover, any other weight of $\L_\a,$ is of the form $$\chi_\a-\a- \sum_{\beta \in\Pi} n_\beta \beta,$$ with $n_\beta\in\N.$
\end{prop}

Fix a subset $\t$ of $\Pi$ and consider $\L_\a:G\to\PGL(V_\a),$ a representation given by Tits's proposition for $\a\in\t.$ Since $\L_\a$ is proximal, one obtains an equivariant map $\xi_\a:\scr F_\t\to\P(V_\a).$

The dual representation $\L_\a^*:G\to\PGL(V_\a^*)$ is also proximal and its highest weight is $\chi_\a\ii.$ Hence, one obtains another equivariant map $\eta_\a=\xi_{\ii\a}:\scr F_{\ii\t}\to\P(V_\a^*).$ Moreover, if $(x,y)\in\posgen_\t$ then $$\eta_\a(x)(\xi_\a(y))\neq0.$$

Consider a scalar product on $V_\a$ such that $\L_\a(K)$ is orthogonal and such that $\Lambda_\a(\exp\frak a)$ is symmetric. The euclidean norm $\|\ \|_\a,$ induced by this scalar product, verifies $$\log\|\L_\a(g)\|_\a=\chi_\a(a(g)),$$ for every $g\in G,$ where $a:G\to\frak a^+$ is the Cartan projection (remark that the operator norm only depends on $\R_+\|\ \|_\a$). 

\begin{lema}[{Quint \cite[Lemma 6.4]{quint1}}]\label{lema:busemanna} For every $\a\in\t$ and $v\in\xi_\a(x)$ one has, $$\chi_\a(\bus_\t(g,x))=\log\frac{\|\L_\a(g)v\|_\a}{\|v\|_\a}.$$
\end{lema}

The set $\{\om_\a|_{\frak a_\t}:a\in\t\}$ is a basis of $\frak a_\t^*,$ and hence so is $\{\chi_\a|_{\frak a_\t}\}_{\a\in\t}.$ We will define the \emph{Gromov product} $\Gr_\t:\posgen_\t\to \frak a_\t,$ by specifying $\chi_\a\circ\Gr_\t$ for every $\a\in\t.$ Define by $$\chi_\a(\Gr_\t(x,y))= \log\frac{|\varphi(v)|}{\|\varphi\|_\a\|v\|_\a},$$ for any $\varphi\in\eta_\a(x)$ and $v\in\xi_\a(y).$

\begin{lema}\label{lema:gromov2} For every $g\in G$ and $(x,y)\in\posgen_\t,$ one has $$\Gr_\t(gx,gy)- \Gr_\t(x,y)=- (\ii\bus_{\ii\t}(g,x)+\bus_\t(g,y)).$$
\end{lema}

\begin{proof} For a norm on a vector space $V,$ every $g\in \PGL(V)$ and every $(\varphi,v) \in\P(V^*)\times\P(V) -\{(\varphi,v)\in\P(V^*)\times\P(V):\varphi(v)=0\}$ one has $$\log\frac{|\varphi\circ
g^{-1}(gv)|}{\|\varphi\circ g^{-1}\|\|gv\|}-\log \frac{|\varphi(v)|}{\|\varphi\|\|v\|}=-\log \frac{\|g\varphi\|}{\|\varphi\|}+\log\frac{\|gv\|}{\|v\|}.$$ The lemma follows from this formula together with the definition of $\Gr_\t$ and Quint's Lemma \ref{lema:busemanna}.
\end{proof}

The following corollary is immediate.

\begin{cor} Let $\rho:\G\to G$ be a Zariski-dense $\t$-convex representation. The function $[\cdot,\cdot]:\bord\G^{(2)}\to\frak a_\t$ defined by $$[x,y]=\Gr_\t(\eta(x),\xi(y)),$$ is a Gromov product for the pair $\{\vect_\t,\vo{\vect_\t}\}.$
\end{cor}

\subsubsection*{Mixing}

We need the following theorem of Benoist \cite{benoist2}:

\begin{teo}[Benoist {\cite[Main Proposition]{benoist2}}]\label{teo:densos} Consider a Zariski-dense subgroup $\grupo$ of $G.$ Then the group generated by $\{\lambda(g):g\in \grupo\}$ is dense in $\frak a.$
\end{teo}

Recall that the Bowen-Margulis measure for $\varphi\in\inte({\cone_\rho^\t}^*),$ is the measure $\Pats_\varphi$ on $\G\/(\bord^2\G\times\frak a_\t),$ induced on the quotient by $$e^{-h_\varphi[\cdot,\cdot]_\varphi}\vo{\mu_\varphi}\otimes\mu_\varphi\otimes\Leb_{\frak a_\t},$$
where $\mu_\varphi$ and $\vo{\mu_\varphi}$ are the Patterson-Sullivan probability measures with cocycles $\varphi\circ\vect_\t$ and $\varphi\circ\vo{\vect_\t},$ respectively. Benoist's theorem (and the continuity of $p_\t$) guarantees the missing hypothesis of Theorem \ref{teo:melange}, applied using $c=\beta_\t,$ and we obtain the following.

\begin{teo}\label{teo:mixingtheta} Let $\rho:\G\to G$ be a Zariski-dense $\t$-convex representation, and consider $\varphi$ in the interior of ${\cone^\t_\rho}^*,$ and $v_0\in\ker\varphi.$ Then there exists an euclidean norm $|\cdot|$ on $\frak a$ such that, for any two compactly supported continuous functions $f_0,f_1:\G\/(\bord^2\G\times\frak a_\t)\to\R$ one has $$(2\pi t)^{(\dim\frak a_\t-1)/2}\Pats_\varphi(f_0\cdot f_1\circ \om_t^{\varphi, v_0})\to e^{-\II(v_0)/2}\Pats_\varphi(f_0) \Pats_\varphi(f_1),$$ as $t\to\infty.$ 
\end{teo}

\noindent (Recall the definition of $\II$ on equation (\ref{eq:I})).

%\begin{obs} When $\t$ is not the full set of simple roots, the last Theorem does not give a mixing property of the Weyl chamber flow of $\rho(\G)$, nevertheless it may turn out to be useful since  $$\rho(\G)\/(\Li_{\rho(\G)}^{(2)} \times\frak a)$$ fibers over $$\G\/(\bord\G^{(2)}\times\frak a_\t)$$ (where $\Li_{\rho(\G)}$ is the limit set of $\rho(\G)$ on $\scr F$).
%\end{obs}

\subsubsection*{The growth indicator function}

Consider a $G$-invariant Riemannian metric on $X,$ and $\|\ \|$ the induced Euclidean norm on $\frak a,$ invariant under the Weyl group. Recall that if $g\in G,$ then  $\|a(g)\|=d_X([K],g[K]).$ Consider a Zariski-dense discrete subgroup $\grupo$ of $G,$ and define by $$h_\grupo=\limsup_{s\to \infty} \frac{\log \#\{g\in \grupo: \|a(g)\|\leq s\}}{s}.$$ Recall that in the introduction we have defined $\psi_\grupo,$ the growth indicator of $\grupo.$ 

\begin{lema}[Quint {\cite[Corollaire 3.1.4]{quint2}}] Let $\grupo$ be a Zariski-dense subgroup of $G,$ then one has $$\sup_{v\in\frak a-\{0\}} \frac{\psi_\grupo(v)} {\|v\|}=h_\grupo.$$ 
\end{lema}

%\begin{teo}[Quint \cite{quint2}]\label{teo:teoquint} Let $\grupo$ be a Zariski-dense discrete subgroup of $G.$ Then $\psi_\grupo$ is concave, upper semi-continuous and the set $$\{v\in\frak a:\psi_\grupo(v)>-\infty\}$$ is the limit cone $\cone_\grupo$ of $\grupo.$ Moreover $\psi_\grupo$ is non negative on $\cone_\grupo$ and positive on its interior.
%\end{teo}
If $\varphi\in \frak a^*$ is such that $\varphi(v)\geq\psi_\grupo(v),$ then $\|\varphi\|\geq h_\grupo.$ One is thus interested in the set $$D_\grupo=\{\varphi\in \frak a^*:\varphi\geq\psi_\grupo\}.$$ Since $\psi_\grupo$ is homogeneous and concave (recall Theorem \ref{teo:teoquint}), the set $D_\grupo$ is convex. The linear form $\ta_\grupo\in D_\grupo$ closest to the origin, is called the \emph{the growth form} of $\grupo,$ and verifies \begin{equation}\label{equation:normaexp}\|\ta_\grupo\|=h_\grupo.\end{equation} 

Again, since $\psi_\grupo$ is concave, and the balls of $\|\ \|$ are strictly convex, one obtains a unique direction $\R_+u_\grupo$ in $\cone_\grupo,$ which realizes the upper bound $$\sup_{v\in\frak a-\{0\}}\frac{\psi_\grupo(v)}{\|v\|},$$ this is called \emph{the growth direction} of $\grupo.$ Choose $u_\grupo$ in the growth direction such that $\ta_\grupo(u_\grupo)=1.$ 

The following remarks are direct consequences of the definitions:

\begin{obs}\label{obs:ortogonal} For every $v\in\R_+ u_\grupo$ one has $|\ta_\grupo(v)|=\|\ta\|\|v\|,$ consequently the hyperplane $\ker\ta_\grupo$ and $\R_+u_\grupo$ are orthogonal for $\|\cdot\|.$ 
\end{obs}

\begin{obs}\label{obs:fuera}The number of elements of $a(\grupo)$ that lie outside a given open cone containing $u_\grupo,$  has exponential growth rate strictly smaller than $h_\grupo.$ 
\end{obs}

A linear form $\varphi\in \frak a^*$ is \emph{tangent} to $\psi_\grupo$ at $x$ if $\varphi\in D_\grupo,$ and $\varphi(x)=\psi_\grupo(x).$ We say that $\psi_\grupo$ has \emph{vertical tangent} at $x,$ if for every $\varphi\in D_\grupo,$ one has $\varphi(x)>\psi_\grupo(x).$ 
Remark that $\ta_\grupo,$ is tangent to $\psi_\grupo,$ in the growth direction $\R_+ u_\grupo.$

Fix from now on a Zariski-dense hyperconvex representation $\rho:\G\to G,$ and denote by $\z:\bord\G\to\scr F$ its $\rho$-equivariant map. The image $\z(\bord\G)$ is the limit set $\Li_{\rho(\G)},$ and thus $$\z\times\z:\bord^2\G\to\Li_{\rho(\G)}^{(2)}$$ is a $\rho$-equivariant H\"older-continuous homeomorphism. Also, the cone $\cone_{\vect^\Pi}$ is the limit cone $\cone_\rho=\cone_{\rho(\G)}$ of $\rho(\G).$ One has the following result.

\begin{teo}[{\cite[Theorem A + Corollary 4.9]{exponential}}]\label{teo:tangente} The growth indicator $\psi_{\rho(\G)},$ of a Zariski-dense hypercon\-vex representation $\rho,$ is strictly concave, analytic on the interior of $\cone_{\rho(\G)},$ and with vertical tangent on the boundary. If $\varphi$ belongs to the interior of $\cone_\rho^*,$ then $h_\varphi\varphi$ is tangent to $\psi_\rho,$ in the dual direction $\uu\varphi.$
\end{teo}

\begin{obs}\label{obs:direction}Hence, for a Zariski-dense hyperconvex representation $\rho$ of $\G,$ the growth direction $\R_+u_{\rho(\G)},$ is the dual direction (in the sense of Section \ref{section:mixingweyl}) of its growth form $\ta_\rho=\ta_{\rho(\G)}$. Moreover, since $\psi_{\rho(\G)}$ has vertical tangent on the boundary of $\cone_{\rho(\G)},$ the growth direction $\R_+ u_{\rho(\G)},$ is contained in the interior of the limit cone.
\end{obs}

Recall that if $\varphi$ is in the interior of $\cone_\rho^*,$ then the H\"older cocycle $\varphi\circ\vect^\Pi$ has finite exponential growth rate. Ledrappier's Theorem \ref{teo:medidas} guarantees the existence of a Patterson-Sullivan probability measure $\mu_\varphi$ on $\bord\G,$ with cocycle $h_\varphi\varphi\circ\vect^\Pi.$ The following corollary is hence direct.

\begin{cor}\label{cor:medidaspaterson}Let $\rho:\G\to G$ be a Zariski-dense hyperconvex representation. For each $\varphi$ tangent to $\psi_{\rho(\G)},$ there exists a unique Patterson-Sullivan measure for $\varphi,$ denoted by $\nu_\varphi.$ Moreover, $\z$ induces an isomorphism between $\mu_\varphi$ and $\nu_\varphi.$
\end{cor}

Consequently, the function $$\z\times\z\times \id:\bord^2\G\times\frak a\to \Li_{\rho(\G)}^{(2)}\times\frak a$$ is a $\rho$-equivariant homeomorphism, and induces on the quotients, a map still denoted by $\z\times\z\times \id:\G\/(\bord^2\G\times\frak a)\to \rho(\G)\/(\Li_{\rho(\G)}^{(2)}\times\frak a),$ which is a measurable isomorphism between the Bowen-Margulis measures of $\varphi$ on each side: $$(\z\times\z\times\id)_*\Pats_\varphi=\Pat_\varphi,$$ where $\Pats_\varphi$ is the Bowen-Margulis measure for $\varphi,$ defined on the introduction.

Theorem \ref{teo:mixingtheta}, together with Remark \ref{obs:direction}, imply the following mixing property of the Weyl chamber flow. Recall that the rank of $G$ is the dimension of $\frak a,$ and that the Weyl chamber flow is the right action by translations of $\exp(\frak a)$ on $\rho(\G)\/G/M.$

\begin{teo}\label{teo:mixing} Let $\rho:\G\to G$ be a Zariski-dense hyperconvex representation. Consider $\varphi$ in the interior of $\cone_\rho$ and consider $v_0 \in\ker\varphi,$ then there exists an euclidean norm $|\cdot|$ on $\frak a,$ such that for all compactly supported continuous functions $f_0,f_1:\rho(\G)\/G/M\to\R,$ one has $$(2\pi t)^{(\rank(G)-1)/2}\Pat_\varphi(f_0\cdot f_1\circ \exp({tu_\varphi+\sqrt tv_0}))$$ converges, as $t\to\infty,$ to $$ce^{-\II(v_0)/2} \Pat_{\varphi}(f_0) \Pat_\varphi(f_1),$$ for a constant $c>0.$ 
\end{teo}

The following corollary will be most useful to us. 

\begin{cor}\label{cor:mixingroblin}Let $\rho:\G\to G$ be a Zariski-dense hyperconvex representation. Then, given two compactly supported continuous functions $f_0,f_1:\rho(\G)\/G/M,$ one has  $$e^{-\|\ta_\rho\|T}\int_{B(0,T)\cap\frak a^+} e^{\ta_{\rho}(u)}\Pat_\TT( f_0\cdot f_1\circ \exp(u))d\Leb_{\frak a}(u)\rightarrow C \Pat_\TT( f_0)\Pat_\TT(f_1),$$ as $T\to\infty,$ for a constant $C>0.$
\end{cor}

The proof of the corollary follows the exact same lines as Thirion \cite[\S 12.k]{thirion1} for Ping-Pong groups. We give a sketch of this proof for completeness.

\begin{proof} In order to simplify notation, denote by $\ta=\TT,$ $H=\ker\ta$ and $u_\rho=\UU.$ Consider the change of variables $G:\R\times H\to\frak a$ given by $$G(t,v)=t\frac{u_\rho}{\|u_\rho\|}+\sqrt t v.$$ It's Jacobian is $(\sqrt t)^{\dim H}=t^{(\rank(G)-1)/2}.$ The integral we are interested in becomes $$\int_He^{-\|\ta\|T}\int_0^\infty e^{\|\ta\|t}(\sqrt t)^{\dim H}\Pat_\ta(f_0\cdot f_1(\exp(G(t,v)))\1_{B(T)}(t,v)dtd\Leb_H(v),$$ where $B(T)=\{(t,v)\in\R\times H:G(t,v)\in \frak a^+,\,\|G(t,v)\|\leq T\},$ and $\1_A(x)$ is the characteristic function.

Recall that $H$ and $u_\rho$ are orthogonal for $\|\cdot\|$ (Remark \ref{obs:ortogonal}). The conditions $t>0$ and $\|G(t,v)\|\leq T$ imply then $$0<t<\frac12(\sqrt{\|v\|^4+4T^2}-\|v\|^2)=R(T,v)$$ Remark that $R(T,v)-T\to -\|v\|^2/2$ as $T\to\infty.$ This, together with Theorem \ref{teo:mixing}, implies that $$e^{-\|\ta\|T}\int_0^{R(T,v)}e^{\|\ta\|t}(\sqrt t)^{\dim H}\Pat_\ta(f_0\cdot f_1(\exp(G(t,v)))\1_{B(T)}(t,v)dt$$ converges to $$ce^{-(\|\ta\|\|v\|^2+\II^{|\cdot|}(\sqrt{\|\ta\|}v))/2}\Pat_\ta(f_0)\Pat_\ta(f_1)$$ as $t\to\infty,$ for a constant $c>0,$ and the euclidean norm $|\cdot|$ (remark that for every $v\in H$ there exists $t_0$ such that, for all $t\geq t_0,$ one has $G(t,v)\in\frak a^+$).

In order to apply the dominated convergence theorem, we need to find an integrable function $F:H\to\R,$ such that for every $v\in H$ one has $$e^{\|\ta\|T}\int_0^{R(T,v)}e^{\|\ta\|t}(\sqrt t)^{\dim H}\Pat_\ta(f_0\cdot f_1(\exp(G(t,v)))\1_{B(T)}(t,v)dt\leq F(v).$$ Remark that, since $\II(v)<0$ for all $v\in H,$ Theorem \ref{teo:mixing} implies that for $t$ large enough, one has $$(\sqrt t)^{\dim H}\Pat_\ta(f_0\cdot f_1(\exp(G(t,v)))\leq K,$$ for a constant $K$ independent of $v.$

Lemma \ref{lema:alpedoN} below, states that there exists a constant $\kappa>0,$ such that for all $(t,v)\in H\times\R,$ with $G(t,v)\in B(T),$ one has $R(T,v)\leq -\kappa\|v\|^2/2.$ Hence $$e^{-\|\ta\|T}\int_0^{R(T,v)}e^{\|\ta\|t}(\sqrt t)^{\dim H}\Pat_\ta(f_0\cdot f_1(\exp(G(t,v)))\1_{B(T)}(t,v)dt \leq$$ $$Ke^{-\|\ta\|T}\int_0^{R(T,v)}e^{\|\ta\|t}\1_{B(T)}(t,v)dt\leq Ke^{\|\ta\|(R(T,v)-T)}\leq K e^{-\kappa\|\ta\|\|v\|^2/2},$$ for a constant $K>0.$ This last function is clearly integrable on $H.$ This finishes the proof.

\end{proof}

\begin{lema}\label{lema:alpedoN} There exists $\kappa>0$ such that, if $(t,v)\in B(T)$ then $R(T,v)\leq-\kappa \|v\|^2/2.$
\end{lema}

\begin{proof} Recall that the angle between the walls of $\frak a^+$ is at most $\pi/2,$ hence, since $u_\rho\in\inte\frak a^+,$ there exists $\t_0\in(0,\pi/2)$ such that if $G(t,v)\in \frak a^+,$ then the angle between $G(t,v)$ and $tu_\rho/\|u_\rho\|$ is at most $\t_0,$ i.e. $$\frac{\|\sqrt t v\|}t\leq\tan(\t_0).$$ From now on, standard computations imply the lemma, see Thirion \cite[page 184]{thirion1} for details.

\end{proof}

%\input{orbitalcounting}

%\begin{appendices}
%  \renewcommand\thetable{\thesection\arabic{table}}
%  \renewcommand\thefigure{\thesection\arabic{figure}}

\section{The orbital counting problem}\label{section:bowenmargulis}

\subsubsection*{General aspects}

The standard reference for this subsection is the book by Guivarc'h-Ji-Taylor \cite{bordeF}. Recall that $G$ is a noncompact real algebraic semisimple group, and $\G$ is the fundamental group of a closed connected negatively curved Riemannian manifold. 

Recall we have denoted by $a:G\to\frak a^+$ the Cartan projection of $G.$ We will define a new projection $\aa:X\times X\to\frak a^+$ by $\aa(g\cdot o,h\cdot o)=a(g^{-1}h).$ Remark that $\aa$ is $G$-invariant, that $\|\aa(p,q)\|=d_X(p,q),$ and that \begin{equation} \label{equation:iii} \ii(\aa(p,q))=\aa(q,p). \end{equation} By definition, one has $q\in K_p \exp(\aa(p,q))\cdot p,$ where $K_p$ is the stabilizer in $G$ of $p.$ 

\begin{obs}\label{obs:asimpt} Remark that there exists $K\in\R_+,$ such that for every $g\in G$ one has $\|\aa(p,g q)-a(g)\|\leq K.$ 
\end{obs}

%If $L$ is a subset of $\frak a^+,$ define the \emph{asymptotic cone} generated by $L,$ by the set of  $z\in \frak a^+$ such that there exists a divergent sequence $\{g_n\}\in L,$ and a sequence $t_n\in\R$ going to $0$ as $n\to\infty,$ with $t_n g_n\to z.$

%We will fix from now on a discrete Zariski dense subgroup $\grupo$ of $G.$ One has the Following Theorem of Benoist \cite{limite1}.

%\begin{teo}[Benoist \cite{limite1}] For $p,q\in X,$ the asymptotic cone generated by $$\{\aa(p,\g q):\g\in\grupo\},$$ is the limit cone $\cone_\grupo$ of $\grupo.$
%\end{teo}

%Observe that, in particular, this asymptotic cone is independent of $p$ and $q.$ 

Similarly (and abusing notation), we will define the \emph{Busemann cocycle} $\buss:\scr F\times X\times X\to\frak a$ by $$\buss(x,g\cdot o,h\cdot o)=\buss_x(g\cdot o,h\cdot o)=\bus(g^{-1},x)-\bus(h^{-1},x).$$ 

A \emph{parametrized flat} is a map $\p:\frak a\to X,$ defined by $\p(v)=g\exp(v)\cdot o,$ for some $g\in G.$ Observe that $G$ acts transitively on parametrized flats, and that the stabilizer of $\p_0:v\mapsto \exp(v)\cdot o$ is the group $M$ of elements in $K$ commuting with $\exp(\frak a).$ We will hence identify the space of parametrized flats with $G/M.$

A \emph{maximal flat} is the image on $X$ of a parametrized flat i.e. the maximal flat associated to $\p,$ is defined by $\mm\p=\p(\frak a)=\{g\exp(v)\cdot o:v\in\frak a\}.$ The space of maximal flats is naturally identified with $G/MA=\posgen$ (recall Hopf's parametrization of $G$ on the Introduction). Denote by $(\wk\zz,\zz):G/M\to\posgen=G/MA$ the canonical projection.

The following proposition is standard.

\begin{prop}[see {\cite[Chapter III]{bordeF}}]\item\begin{enumerate}\item Let $\p,{\sf g}$ be two parametrized flats, then the function $\frak a\to\R,$ defined by $$v\mapsto d_X(\p(v),{\sf g}(v)),$$ is bounded on the Weyl chamber $\frak a^+,$ if and only if $\zz(\p)=\zz({\sf g}).$\item A pair $(p,x)\in X\times\scr F$ determines a unique parametrized flat $\p,$ such that $\p(0)=p$ and $\zz(\p)=x.$ \item A point $(x,y)\in\posgen$ determines a unique maximal flat $\mm\p$ such that $\wk\zz(\p)=x$ and $\zz(\p)=y.$ 
\end{enumerate}
\end{prop}

The usual relation between the Cartan projection and Busemann's cocycle, is given the following lemma of Quint \cite{quint1}. Observe that if $p,q\in X$ are such that $\aa(p,q)\in\inte(\frak a^+),$ then there is a unique parametrized flat $\p_{pq}$ such that $\p_{pq}(0)=p,$ and $\p_{pq}(\aa(p,q))=q.$ Denote by $x_{pq}=\zz(\p_{pq}),$ and recall that $\Pi$ is the set of simple roots of $G.$

\begin{lema}[{Quint \cite[Lemma 6.6]{quint1}}]\label{lema:bus} Fix $p,q\in X,$ then $$\aa(p,z)- \aa(q,z) -\buss_{x_{pz}}(p,q)\to0,$$ as $\min_{\a\in\Pi}\a(\aa(p,z))\to\infty.$
\end{lema}

Given $r>0,$ define the \emph{shadow (on $\scr F$) of $q$ seen from $p$ of size $r,$} by $$\cal O_r(p,q)=\{\zz(\p):\p\in G/M,\ \p(0)=p,\ \exists v\in\inte(\frak a^+),\;d_X(\p(v),q)<r\}.$$ Denote by $B_X(p,r),$ the ball on $X$ of radius $r$ centered at $p,$ and define by $$\cal O^+_r(p,q)=\bigcup_{p_0\in B_X(p,r)}\cal O(p_0,q),$$ and $$\cal O^-_r(p,q)=\bigcap_{p_0\in B_X(p,r)}\cal O(p_0,q).$$ Finally, for $x\in\scr F$ define the \emph{shadow of $q$ seen from $x$ of size $r,$} by $$\cal O_r(x,q)=\{\zz(\p):\p\in G/M,\ d_X(\p(0),q)<r,\ \wk\zz( \p)=x\}.$$

\begin{lema}[{Thirion \cite[Proposition 8.66]{thirion1}}]\label{lema:buseman-cartan} Given $r,\eps>0$ there exists $R>0,$ such that if $d_X(p,q)\geq R,$ and $x\in \cal O_r^+(p,q),$ then $$\|\buss_x(p,q)-\aa(p,q)\|\leq\eps.$$
\end{lema}

Let $\grupo$ be a Zariski-dense discrete subgroup of $G,$ and consider a linear form $\varphi\in\frak a^*,$ tangent to $\psi_\grupo$ on a direction in the interior of $\cone_\grupo.$ Denote by $\nu_\varphi$ the $\varphi$-Patterson-Sullivan measure (recall Quint's Theorem \ref{teo:JF-PS}). Define the $\varphi$-\emph{Patterson-Sullivan density} $\{\mu_p\}_{p\in X}$ by $\mu_o=\nu_\varphi$ and $$\frac{d\dens_p}{d\dens_o}(x)=e^{-\varphi(\buss_x(p,o))}.$$

Since $\scr F$ is $K_p$-homogeneous and $K_p$ is compact, there is a unique $K_p$-invariant probability measure on $\scr F.$ This gives an embedding of $X$ on the space of probability measures $\cal M(\scr F).$ The closure of this embedding, denoted by $\F,$ is called \emph{the Furstenberg compactification of $X$.} Observe that if $v\in\inte(\frak a^+)$ and $k\in K,$  then $$ke^{tv}\cdot o\to \delta_{kM},$$ as $t\to\infty,$ on $\cal M(\scr F),$ where $\delta_{kM}$ is the Dirac delta on $kM.$

A pair $(p,x)\in X\times\scr F$ is \emph{in good position} (w.r.t. $\grupo$), if the parametrized flat determined by $p$ and $x$ verifies $\wk\zz(\p)\in\Li_\grupo.$ Given $a,b\in\R$ and $\eps>0,$ we will say that $a\simbolo\eps b$ if $e^{-\eps}a\leq b\leq e^\eps a.$

\begin{lema}[{Thirion \cite[Lemma 10.7]{thirion1}}]\label{lema:sombras} Fix a pair in good position $(p,x)\in X\times\scr F.$ Then for every $\eps,r>0$ there exists $0<r_1<r,$ and a neighborhood $V_x$ of $x$ on $\F,$ such that for all $z\in V_x$ one has $$\dens_p(\cal O_{r_1}(z,p))\simbolo\eps\dens_p(\cal O_{r_1}(x,p)).$$ 
\end{lema}

If $p=g\cdot o\in X,$ define the $\varphi$-\emph{Gromov product} (or simply \emph{Gromov product}) \emph{based on} $p,$ ${\Gromov\cdot\cdot}_p={\Gromov\cdot\cdot}_p^\varphi:\posgen\to\R,$ by $${\Gromov xy}_{g\cdot o}=\varphi(\Gr_\Pi(g^{-1}x,g^{-1}y)).$$ 

\begin{obs}\label{obs:gromov0} Remark that ${\Gromov\cdot\cdot}_p$ is continuous, and that if $p$ belongs to the maximal flat determined by $(x,y)\in\posgen,$ then $[x,y]_p =0.$ 
\end{obs}

Denote by $\{\vo\dens_p\}_{p\in X}$ the $\varphi\circ\ii$-Patterson-Sullivan density of $\grupo.$ The following remark follows from the definitions.

\begin{obs}\label{obs:punto} The measure $e^{{-\Gromov\cdot\cdot}_p} \vo\dens_p\otimes \dens_p\otimes\Leb_{\frak a}$ is independent of $p.$ This is the $\varphi$-\emph{Bowen-Margulis measure} of $\grupo,$ and is denoted by $\w\chi_\varphi.$
\end{obs}

%\begin{lema}[{Quint \cite[Lemma 8.9]{quint1}}]\label{lema:amigazo} Consider $\varphi\in\frak a^*$ such that $\varphi|_{\frak a^+-\{0\}}\geq0$ and consider $v\in \frak a^+$ such that $\varphi(w)=\<v,w\>$ for all $w\in \frak a.$ Then $$\|v\| B_{\pi_v(x)}(p,q)=\varphi(\buss_x(p,q)).$$
%\end{lema}

\subsection*{The main theorem}

This section is completely devoted to the proof of the following theorem. The method is that of Roblin \cite{roblin} which was also adapted to some higher rank situations, by Thirion \cite{thirion1}. 

If $\rho:\G\to G$ is a Zariski-dense hyperconvex representation, denote by $\ta=\ta_\rho=\ta_{\rho(\Gamma)}$ its growth form. Recall that $\|\ta\|=h_{\rho(\G)},$ and that $\ta$ is $\ii$-invariant, and tangent to the growth indicator of $\rho(\G).$ Denote by $\{\mu_p\}_{p\in X}$ the $\ta$-Patterson-Sullivan density of $\rho(\G).$ 

\begin{teo}\label{teo:quiensabe}Let $\rho:\G\to G$ be a Zariski-dense hyperconvex representation, and consider $p,q\in X,$ then $$e^{-\|\ta\|T}\sum_{\g\in\G: d_X(p,\rho(\g) q)\leq T}\delta_{\rho(\g) q}\otimes\delta_{\rho(\g^{-1}) p}\to c\mu_p\otimes\mu_q,$$ for the weak-star convergence on $C(\F\times\F),$ as $T\to\infty,$ for a constant $c=c(p,q)>0.$
\end{teo}

A Zariski-dense hyperconvex representation $\rho:\G\to G$ is fixed from now on. In order to simplify notation, we will identify $\G$ with $\rho(\G),$ i.e. if $p\in X$ then $\g p$ means $\rho(\g) p.$ 

For $T\in\R_+,$ define the measure $\LL^T(p,q)$ on $\F\times\F,$ by $$\LL^T(p,q)=e^{-\|\ta\|T}\sum_{\g\in\G:d_X(p,\g q)\leq T} \delta_{\g q}\otimes \delta_{\g^{-1}p}.$$

%Denote also by ${\displaystyle B_\frak a(\scr A,r)=\bigcup_{v\in\scr A}B_\frak a(v,r),}$ the $r$-neighborhood of $\scr A$ on $\frak a.$

If $A$ is a subset of $\scr F,$ and $r>0,$ consider the subset $C^+_r(p,A)$ of $X,$ defined by the $r$-neighborhood of $$\{\p(\frak a^+): \p\in G/M,\,d_X(\p(0),p)\leq r,\ \zz(\p)\in A\},$$ and consider the set $C^-_r(p,A)$ defined by $$\{y\in X: B_X(y,r)\subset\bigcap_{\{q\in X:d_X(q,p)\leq r\}}\bigcup_{\{\p\in G/M:\p(0)=q,\ \zz(\p)\in A\}} \p(\frak a^+)\}.$$

The following proposition, is the main step of the proof of Theorem \ref{teo:quiensabe}.

\begin{prop}\label{tutti} Consider $p,q\in X$ and $x,y\in\scr F,$ such that $(p,x)$ and $(q,y)$ are in good position. Then, for every $\eps>0,$ there exists a neighborhood $W$ of $(x,y)$ on $\F\times\F,$ such that for every Borel sets $A,B\subset\scr F$ with $A\times B\subset W,$ one has $$\limsup_{T\to\infty} \LL^T(p,q)(C^-_1(p,A)\times C^-_1(q,B))\leq e^\eps \ca\dens_p(A)\vov\dens_q(B)$$ and $$\liminf_{T\to\infty} \LL^T(p,q)(C^+_1(p,A)\times C^+_1(q,B))\geq e^{-\eps}\ca \dens_p(A)\vov\dens_q(B),$$ for a constant $c>0.$
\end{prop}

\begin{proof} For a maximal flat $\mm\p$ and $p\in X,$ denote by $\p^p$ the parametrized flat such that $[\p^p]=[\p],$ and such that $\p^p(0)$ is the orthogonal projection of $p$ on $\mm\p.$ If $v\in\frak a,$ denote by $\tran_v:\frak a\to\frak a$ the translation by $v,$ i.e. $\tran_v(u)=u+v.$ For a given subset $A$ of $\scr F,$ consider the subsets of $G/M$ defined by $$K_r^+(p,A)=\{\p^p\circ\tran_v:v\in B_\frak a(0,r),\ d_X(\p^p(0),p)\leq r,\ \zz(\p)\in A\},$$ and $$K_r^-(p,A)=\{\p^p\circ\tran_v:v\in B_\frak a(0,r),\ d_X(\p^p(0),p)\leq r,\ \wk\zz(\p)\in A \},$$ where $B_{\frak a}(0,r)$ is the ball on $\frak a$ of radius $r,$ centered at $0,$ for the euclidean norm $\|\cdot\|.$

Denote by $V_x$ and $V_y$ neighborhoods on $\F$ of $x$ and $y$ respectively, given by Lemma \ref{lema:sombras} for $(p,x)$ and $(q,y).$ We can assume that Lemma \ref{lema:buseman-cartan} holds, for $p$ and every point of $V_x\cap X,$ and for $q$ and every point of $ V_y\cap X.$ We will show that $V_x\times V_y$ is the desired neighborhood. Consider then $A,B$ Borel subsets of $\scr F$ such that $A\times B\subset V_x\times V_y.$ Let us simplify notation and denote by $K^+=K_r^+(p,A)$ and $K^-=K^-_r(q,B).$

Given $\g\in\G$ and $T>0,$ define $\Xi(\g,T)$ by $$\Xi(\g,T)=\int_{B_\frak a(0,T)\cap\frak a^+}e^{\ta(v)}\BM(K^+\cdot \exp(v)\cap \g\cdot K^-)d\Leb_{\frak a}(v).$$ Following Roblin's \cite{roblin} method (see also Thirion \cite{thirion1}), we will compute $$e^{-\|\ta\|T}\sum_{\g\in\G}\Xi(\g,T),$$ in two different ways. Observe first that Corollary \ref{cor:mixingroblin} gives $$e^{-\|\ta\|T }\sum_{\g\in\G}\Xi(\g,T)\to \ca \BM(K^+)\BM(K^-),$$ as $T\to\infty,$ for a constant $c>0.$ Let's compute then $\BM(K^+)$ and $\BM(K^-).$ Remark \ref{obs:punto} states that $$\BM=e^{-{\Gromov\cdot\cdot}_p} \vov\dens_p\otimes \dens_p\otimes\Leb_{\frak a},$$ hence $$\BM(K^+)=\int_{\frak a}\int_A\int_{\cal O_r(z,p)}e^{-{\Gromov wz}_p}\1_{B_\frak a(0,r)}(v)d\vov\dens_p(w)d\dens_p(z)d\Leb_\frak a(v).$$ 

\noindent
Since $w\in \cal O_r(z,p),$ Remark \ref{obs:gromov0} implies that $e^{-{\Gromov wz}_p}\simbolo\eps1,$ and thus $$\BM(K^+)\simbolo{\eps}\int_{\frak a}\int_A\vov\dens_p(\cal O_r(z,p))\1_{B(0,r)}(v)d\dens_p(z)d\Leb_\frak a(v)$$ $$=r^{\rank(G)}\int_A\vov\dens_p(\cal O_r(z,p))d\dens_p(z).$$ Since $z\in A\subset V_x,$ Lemma \ref{lema:sombras} applied to the pair $(p,x),$ implies that $\vov\dens_p(\cal O_r(z,p))\simbolo\eps\vov\dens_p(\cal O_r(x,p)),$ and hence $$\BM(K^+)\simbolo{2\eps} r^{\rank(G)}\vov\dens_p(\cal O_r(x,p))\dens_p(A).$$ Analogous reasoning, using the equality $\BM=e^{-{\Gromov\cdot\cdot}_q} \vov\dens_q\otimes \dens_q\otimes\Leb_{\frak a},$ gives $\BM(K^-)\simbolo\eps r^{\rank(G)}\dens_q(\cal O_r(y,q))\vov\dens_q(B).$ Hence, if we denote by $${\sf H}=r^{2\rank(G)}\dens_q(\cal O_r(y,q))\vov\dens_p(\cal O_r(x,p)),$$ one has \begin{equation}\label{equation:mix}e^{-\|\ta\|T} \sum_{\g\in\G} \Xi(\g,T) \simbolo{4\eps}\ca\dens_p(A)\vov\dens_q(B){\sf H},\end{equation} for all big enough $T.$ Remark that, since $(p,x)$ and $(q,y)$ are in good position, one has ${\sf H}\neq0.$ This will allow us later to divide by ${\sf H}.$

We will now explicitly compute $\sum_{\g\in\G}\Xi(\g,T).$ Observe that, if $d_X( p,\g q)\geq T-r$ then, for all $v\in B_\frak a(0,T)\cap \frak a^+,$ one has $$K^+\cdot\exp(v)\cap \gamma\cdot K^-=\vacio,$$ and hence $\Xi(\g,T)=0.$ Thus, $$\sum_{\g\in\G}\Xi(\g,T-r)=\sum_{\g\in\G: d_X(p,\g q)\leq T}\Xi(\g,T-r)$$

\begin{afi}\label{afi:hecho} Denote by $V_A^+=C^+_r(p,A)\cap V_x$ and $V^+_B=C^+_r(q,B)\cap V_y,$ then there exist constants $C$ and $K,$ such that for all big enough $T$ one has $$\sum_{\g\in\G}\Xi(\g,T-r)\leq C+ e^{K\eps}{\sf H}\sum \1_{V_A^+}(\g q)\1_{V_B^+}(\g^{-1} p),$$ where the sum is over all $\g\in\G$ such that, $d_X(p,\g q)\leq T$ and $(\g q,\g^{-1}p)\in  C^+_r(p,A)\times C^+_r(q,B).$ Moreover, define by $V_A^-=C^-_r(p,A)\cap V_x$ and $V^-_B=C^-_r(q,B)\cap V_y,$ then $$\sum_{\g\in\G} \Xi(\g,T+r)\geq -C+e^{-K\eps}{\sf H}\sum \1_{V_A^-}(\g q)\1_{V_B^-}(\g^{-1} p),$$ where the sum is over all $\g\in\G$ such that, $d_X(p,\g q)\leq T$ and $(\g q,\g^{-1}p)\in  C^+_r(p,A)\times C^+_r(q,B).$
\end{afi}

\begin{proof} We will only show the upper bound (the lower bound being analogous). Observe that, if for some $\g\in\G$ one has $\Xi(\g,T-r)\neq0,$ then $K^+\cdot\exp(v)\cap \gamma\cdot K^-\neq\vacio,$ for some $v\in B_{\frak a}(0,T)\cap\frak a^+.$ This intersection is contained in $$(\cal O^+_r(\g q,p)\cap \g B)\times (\cal O_r^+(p,\g q)\cap A)\times\frak a,$$ and necessarily one has:\begin{enumerate}\item[i)] $v\in B_\frak a(\aa(p,\g q),r),$ i.e. $d_X(p,\g q)\leq T,$ and \item[ii)] $(\g q,\g^{-1}p)\in C^+_1(p,A)\times C^+_1(q,B).$ \end{enumerate}

Using the explicit definition of $\BM,$ one obtains $\sum_{\g\in\G}\Xi(\g,T)\leq $ $$C+ e^{K_0\eps}r^{2\rank(G)}\sum e^{\ta(\aa(p,\g q))}\vov\dens_p(\cal O^+_r(\g q,p))\dens_p(\cal O_r^+(p,\g q))\1_{V_A}(\g q)\1_{V_B}(\g^{-1} p),$$ where the sum is over all $\g\in\G$ that verify i) and ii) above, and $C$ is a constant independent on $T,$ determined by the (finitely many) $\g\in\G$ such that $$(\g q,\g^{-1}p)\in C^+_r(p,A)\times C^+_r(q,B)-V_x\times V_y.$$ Since $\g q\in V_x$ one has $\vov\dens_p(\cal O^+_r(\g q,p))\simbolo\eps\vov\dens_p(\cal O_r(x,p)),$ and the last equation becomes $$C+ e^{K_0\eps}r^{2\rank(G)}\vov\dens_p(\cal O_r(x,p))\sum e^{\ta(\aa(p,\g q))}\dens_p(\cal O_r^+(p,\g q))\1_{V_A}(\g q)\1_{V_B}(\g^{-1} p).$$

Using Lemma \ref{lema:buseman-cartan}, and the fact that $\g q$ belongs to $V_A^+\subset V_x,$ one obtains $$e^{\ta(\aa(p,\g q))}\simbolo{\eps}e^{\ta(\buss_z(p,\g q))},$$ for any $z\in \cal O_r^+(p, \g q).$  Applying the definition of $\{\dens_m\}_{m\in X},$ one has that  $$e^{\ta(\aa(p,\g q))}\dens_p(\cal O_r^+(p,\g q))\simbolo{\eps}e^{\ta(\buss_z(p,\g q))}\dens_p(\cal O_r^+(p,\g q))\simbolo\eps\dens_{\g q}(\cal O_r^+(p,\g q))$$ $$=\dens_q(\cal O_r^+(\g^{-1}p, q))\simbolo{\eps}\dens_q(\cal O^+_r(y,q)),$$ since $\g^{-1} p\in V^+_B \subset V_y.$ Hence, $$\sum_{\g\in\G}\Xi(\g,T-r) \leq C+ e^{K\eps}{\sf H} \sum \1_{V_A^+}(\g q)\1_{V_B^+}(\g^{-1} p),$$ where the sum is over all $\g\in\G$ that verify i) and ii) above. This finishes the proof of the claim.
\end{proof}

The proof of the proposition will be completed when we compute $$e^{-\|\ta\|T}\sum_{\g\in\G}\Xi(\g,T),$$ assembling equation (\ref{equation:mix}) and Claim \ref{afi:hecho}. For all big enough $T,$ one has $$e^{-4\eps}c\dens_p(A) \vov\dens_q(B){\sf H}\leq e^{-\|\ta\|T}\sum_{\g\in\G}\Xi(\g,T)\leq$$  $$e^{-\|\ta\|T}(C+e^{K\eps}{\sf H}\sum \1_{V_A^+}(\g q)\1_{V_B^+}(\g^{-1} p)),$$ for some $K$ and $C$ independent of $T,$ where the sum is  over all $\g\in\G$ that verify i) and ii) above. Since $C$ is independent of $T$ and since ${\sf H}\neq0,$ one obtains  $$\liminf_{T\to\infty} \LL^T(p,q)(C^+_1(p,A)\times C^+_1(q,B))\geq e^{-K_0\eps}\ca \dens_p(A)\vov{\dens}_q(B).$$ The other inequality follows similarly.%Making $r\to0$ one obtains $$\liminf_{T\to\infty} \LL^T(p,q,\scr C)(C^+_1(p,\scr C,A)\times C^+_1(q,\scr C,B))\geq e^{-K_0\eps}\ca(\varphi) \dens_p(A)\vo{\dens}_q(B).$$ 
\end{proof}

We can now prove Theorem \ref{teo:quiensabe}. For $\scr A\subset\frak a^+,$ define the measure $\LL^T(p,q,\scr A)$ on $\F\times\F,$ by $$\LL^T(p,q,\scr A)=e^{-\|\ta\|T}\sum_{\g\in\G: \aa(p,\g q)\in\scr A,\; d_X(p,\g q)\leq T} \delta_{\g q}\otimes \delta_{\g^{-1}p}.$$ Remark that $\LL^T(p,q)=\LL^T(p,q,\frak a^+).$

We will need the following lemma.

\begin{lema}\label{lema:afuera} Let $\grupo$ be a Zariski-dense subgroup of $G.$ Consider a continuous function $f:\F\times\F\to\R,$ and an open cone $\scr C$ with $u_\grupo\in\scr C.$ Then $$e^{-h_\grupo t}\sum f(g q,g^{-1}p)\to0$$ as $t\to\infty,$ where the sum is over all $g\in\grupo$ such that  $d_X(p,gq)\leq t$ and $\aa(p,g q)\notin\scr C.$
\end{lema}

\begin{proof}The lemma follows directly from Remark \ref{obs:fuera}, together with Remark \ref{obs:asimpt}.
\end{proof}

Consequently one has the following.
\begin{samepage}
\begin{lema}\label{lema:novo}  Let $\scr C\subset \frak a^+$ be an open cone with $u_\G\in \scr C,$ then $$\LL^T(p,q,\scr C)-\LL^T(p,q)\to0,$$ for the weak-star convergence on $C(\F\times\F),$ as $T\to\infty.$
\end{lema}
\begin{flushright}$\square$\end{flushright}
\end{samepage}

\begin{proof}[Proof of Theorem \ref{teo:quiensabe}] It remains to overpass the good position hypothesis on Proposition \ref{tutti}.

Remark that if $x\in\scr F,$ then one can choose $z\in\bord\G$ such that $(x,\z(z))\in\posgen.$ Fix then $(x,y)$ in $\scr F\times\scr F,$ and consider $(z,w)\in\z(\bord\G)^2$ such that $(x,z)$ and $(y,w)$ belong to $\posgen.$ Choosing $p_0$ on the maximal flat determined by $(x,z),$ and $q_0$ on the maximal flat determined by $(y,w),$ one gets that $(p_0,x)$ and $(q_0,y)$ are both in good position. 

Applying Proposition \ref{tutti} to the pairs $(p_0,x)$ and $(q_0,y),$ and a given $\eps>0,$ one obtains a neighborhood $W$ of $(x,y)\in\F^2$ such that if $A\times B$ is a Borel set contained in $\scr F^2\cap W,$ then \begin{equation}\label{equation:tutti2}\liminf_{T\to\infty} \LL^T(p_0,q_0)(C^+_1(p_0,A)\times C^+_1(q_0,B)) \geq e^{-K_0\eps}\ca \dens_{p_0}(A)\vov{\dens}_{q_0}(B).\end{equation}

If $W$ is small enough, then for every $(s,t)\in X^2\cap W,$ Quint's Lemma \ref{lema:bus} implies that $$\|\aa(p_0,s)-\aa(p,s)-\buss_x(p_0,p)\|\leq \eps,$$ and $$\|\aa(q_0,t)-\aa(q,t)-\buss_y(q_0,q)\|\leq \eps.$$

Discarding finitely many $\g\in\G,$ we can assume that if $\g q_0\in  C^+_1(p_0,A)$ and $\g^{-1}p\in C^+_1(q,B),$ then $(\g q_0,\g^{-1} p)\in  W.$ The last inequalities then give $$\|\aa(p_0,\g q_0)-\aa(p,\g q_0)-\buss_x(p_0,p)\|\leq \eps$$ and  $$\|\aa(q_0,\g^{-1}p)-\aa(q,\g^{-1}p)-\buss_y(q_0,q)\|\leq \eps.$$

Equation (\ref{equation:iii}), and $G$-invariance of $\aa,$ implies that $\aa(q_0,\g^{-1}p)=\ii(\aa(p,\g q_0)),$ and since $\ii^2=\id$ one has $$\|\aa(p,\g q_0)-\aa(p,\g q)-\ii\buss_y(q_0,q)\|<\eps.$$

\noindent
Consequently, $$\|\aa(p_0,\g q_0)-\aa(p,\g q)-(\buss_x(p_0,p)+\ii\buss_y(q_0,q))\|\leq 2\eps.$$ Hence, $$\ta(\aa(p_0,\g q_0))\leq \ta(\aa(p,\g q))+\ta (\buss_x(p_0,p) +\ii \buss_y(q_0,q)) +\delta,$$ for some $\delta$ ($\ta$ is continuous at 0).

Recall that if $v\in \R\cdot u_{\rho(\G)}$ then $|\ta(v)|=\|\ta\|\|v\|$ (Remark \ref{obs:ortogonal}). Consider then a closed cone $\scr C,$ with $u_{\rho(\G)}\in\inte\scr C,$ such that for all $v\in\scr C$ one has $$\ta(v)\simbolo\eps\|\ta\|\|v\|.$$ Remark that, since $\aa(p_0,\g q_0)$ is at bounded distance from $\aa(p,\g q),$ we can consider an open cone $\scr C'$ with $u_{\rho(\G)}\in\scr C',$ such that if $\g$ is big enough, and $\aa(p,\g q)\in\scr C',$ then $\aa(p_0,\g q_0)\in\scr C.$ Hence, for all big enough $\g\in\G$ such that $\aa(p,\g q)\in \scr C'$ one has $$d_X(p_0,\g q_0)\leq d_X(p,\g q)+\frac1{\|\ta\|}(\ta (\buss_x(p_0,p) +\ii (\buss_y(q_0,q))) +\delta).$$

\noindent Lemma \ref{lema:novo} together with equation (\ref{equation:tutti2}), imply that $$ \liminf_{T\to\infty} \LL^T(p_0,q_0,\scr C)(C^+_r(p_0,A)\times C^+_r(q_0,B)\geq e^{-\eps} \ca\dens_{p_0}(A)\vov\dens_{q_0}(B).$$

\noindent
Denoting by $$T'=T+\frac1{\|\ta\|}\ta(\buss_x(p_0,p)+\ii(\buss_y(q_0,q)))+\delta,$$ one concludes that $\LL^T(p,q)( C^+_r(p,A)\times C^+_r(q,B))\geq \LL^T(p,q,\scr C')(C^+_r(p,A)\times C^+_r(q,B))\geq$ $$ e^{\ta(\buss_x(p_0,p)+ \ii(\buss_y(q_0,q)))+ \delta}\LL^{T'}(p_0,q_0,\scr C)(C^+_r(p_0,A)\times C^+_r(q_0,B)).$$ Thus, $\lim\inf_T \LL^T(p,q)( C^+_r(p,A)\times C^+_r(q,B))\geq$ $$ e^{-\eps} \ca e^{\ta(\buss_x(p_0,p)+\ii\buss_y(q_0,q)))}\dens_{p_0}(A)\vov\dens_{q_0}(B).$$ Finally, by definition of $\{\mu_m\}_{m\in X},$ one has $$e^{\ta(\buss_x(p_0,p))}\dens_{p_0}(A)\simbolo\eps \dens_p(A),$$ and $$e^{\ta(\ii\buss_y(q_0,q))}\vov\dens_{q_0}(B)\simbolo\eps \vov\dens_q(B).$$ One concludes that $$\liminf_{T\to\infty} \LL^T(p,q)(C^+_r(p,A)\times C^+_r(q,B))\geq e^{-\eps} \ca\dens_p(A)\vov\dens_q(B),$$ as desired. The other inequality is analogous, and a standard partition of unity argument finishes the proof of the theorem.
\end{proof}

%\end{appendices}

\bibliography{stage}
\bibliographystyle{plain}

\author{$\ $ \\
Andr\'es Sambarino\\
 Institut de Math\'ematiques de Jussieu\\
Universit\'e Pierre et Marie Curie\\
  4 place Jussieu, 75252 Paris Cedex\\
  \texttt{andres.sambarino@gmail.com}}

\end{document}